\documentclass[a4paper,10pt]{amsart}

\usepackage[utf8]{inputenc} 

\usepackage{amsmath}
\usepackage{amstext}
\usepackage{amsfonts}
\usepackage{amsthm}
\usepackage{amssymb}
\usepackage{enumitem}

\usepackage{graphicx}
\usepackage{dsfont}

\usepackage{hyperref}
\hypersetup{colorlinks}

\usepackage{xcolor}

\numberwithin{equation}{section}

\newcommand{\R}{{\mathbf{R}}}

\newcommand{\E}{{\mathbf{E}}}
\newcommand{\N}{{\mathbf{N}}}

\newcommand{\D}{{\mathcal{D}}}
\newcommand{\F}{{\mathcal{F}}} 
\renewcommand{\P}{{\mathbf{P}}} 

\renewcommand{\L}{{\mathcal{L}}}

\newcommand{\X}{{\mathcal{X}}}

\newcommand{\G}{{\mathcal{G}}}

\newcommand{\diff}[1]{\,\mathrm{d}#1}
\newcommand{\inner}[3][]{\langle#2,#3\rangle_{#1}}

\newcommand{\id}{\mathrm{id}}
\newcommand{\triple}{{\vert\kern-0.25ex\vert\kern-0.25ex\vert}}
\newcommand{\Eta}{\mathcal{H}}

\theoremstyle{plain}

\newtheorem{definition}{Definition}[section]
\newtheorem{theorem}[definition]{Theorem}
\newtheorem{lemma}[definition]{Lemma}

\newtheorem{assumption}[definition]{Assumption}

\theoremstyle{definition}
\newtheorem{remark}[definition]{Remark}

\begin{document}

\title[Backward Euler method for multi-valued SDE
]
{
Error estimates of\\
the backward Euler--Maruyama method for\\
multi-valued stochastic differential equations\\
}

\author[M.~Eisenmann]{Monika Eisenmann}
\address{Monika Eisenmann\\
Technische Universit\"at Berlin\\
Institut f\"ur Mathematik, Secr. MA 5--3\\
Stra\ss e des 17.~Juni 136\\
DE--10623 Berlin\\
Germany}
\email{meisenma@math.tu-berlin.de}

\author[M.~Kov\'acs]{Mih\'aly Kov\'acs}
\address{Mih\'aly Kov\'acs\\
Faculty of Information Technology and Bionics\\
P\'azm\'any P\'eter Catholic University\\
P.O. Box 278, Budapest\\
Hungary}
\email{kovacs.mihaly@itk.ppke.hu}

\author[R.~Kruse]{Raphael Kruse}
\address{Raphael Kruse\\
Technische Universit\"at Berlin\\
Institut f\"ur Mathematik, Secr. MA 5--3\\
Stra\ss e des 17.~Juni 136\\
DE--10623 Berlin\\
Germany}
\email{kruse@math.tu-berlin.de}

\author[S.~Larsson]{Stig Larsson}
\address{Stig Larsson\\
Department of Mathematical Sciences\\
Chalmers University of Technology and University of Gothenburg\\
SE--412 96 Gothenburg\\
Sweden}
\email{stig@chalmers.se}

\keywords{multi-valued stochastic differential equation, 
backward Euler--Maruyama method, strong convergence, stochastic gradient flow,
discontinuous drift, H\"older continuous drift} 
\subjclass[2010]{65C30, 60H10, 34A60}

\begin{abstract}
  In this paper, we derive error estimates of the backward Euler--Maruyama
  method applied to multi-valued stochastic differential equations. An important
  example of such an equation is a stochastic gradient flow whose associated
  potential is not continuously differentiable, but assumed to be convex.
  We show that the backward Euler--Maruyama method is well-defined and
  convergent of order at least $1/4$ with respect to the root-mean-square norm.
  Our error analysis relies on techniques for deterministic problems 
  developed in [Nochetto, Savar\'e, and Verdi, Comm.\ Pure Appl.\ Math., 2000]. 
  We verify that our setting applies to an overdamped Langevin equation
  with a discontinuous gradient and to a spatially semi-discrete approximation
  of the stochastic $p$-Laplace equation.
\end{abstract}

\maketitle

\section{Introduction}
\label{sec:intro}

In this paper, we investigate the numerical approximation of 
multi-valued stochastic differential equations (MSDE). An important example
of such equations is provided by stochastic gradient flows with a convex
potential. More precisely, let $T \in (0,\infty)$ and $(\Omega,\F, (\F_t)_{t \in [0,T]},
\P)$ be a filtered probability space 
satisfying the usual conditions. By $W \colon [0,T] \times \Omega \to \R^m$, $m
\in \N$, we denote a standard $(\F_t)_{t \in 
[0,T]}$-adapted Wiener process. For instance, let us consider the numerical
treatment of nonlinear, overdamped Langevin-type equations of the form
\begin{align}
  \label{eq:langevin}
  \begin{cases}
    \diff{X(t)} = - \nabla \Phi(X(t)) \diff{t}  + g_0 \diff{W(t)},
    \quad t \in (0,T],\\
    X(0) = X_0,
  \end{cases}
\end{align}
where $X_0 \in L^p(\Omega,\F_0,\P;\R^d)$, $p \in [2,\infty)$, $g_0 \in
\R^{d,m}$, and $\Phi \colon \R^d \to \R$ are given.
These equations have many important applications, for example, in Bayesian
statistics and molecular dynamics. We refer to \cite{durmus2016,
leimkuhler2016, lelievre2010, roberts1996, sabanis2018} and the references
therein. 

We recall that, if the gradient $\nabla \Phi$ is of superlinear
growth, then the classical forward Euler--Maruyama method is known to be 
divergent in
the strong and weak sense, see \cite{hutzenthaler2011}. This problem
can be circumvented by using modified versions of the explicit
Euler--Maruyama method based on techniques such as taming,
truncating, stopping, projecting, or adaptive strategies, cf.~\cite{beyn2016,
brosse2018, hutzenthaler2015, kelly2018, mao2016, sabanis2016}. 

In this paper, we take an alternative approach by considering
the backward Euler--Maruyama method. Our main motivation for considering 
this method lies in its good stability properties, which
allow its application to stiff problems arising, for instance, from the spatial
semi-discretization of stochastic partial differential equations.
Implicit methods have also been studied extensively in the context of
stochastic differential equations with superlinearly growing coefficients. For
example, see \cite{andersson2017, higham2002, hu1996, mao2013, mao2013a}.

The error analysis in the above mentioned papers on explicit and implicit
methods typically requires a certain
degree of smoothness of $\nabla \Phi$ such as local Lipschitz continuity. 
The purpose of this paper is to derive error estimates of the backward
Euler--Maruyama method for equations of the form \eqref{eq:langevin},
where the associated potential $\Phi \colon \R^d \to \R$ is not necessarily
continuously differentiable, but assumed to be convex. 

For the formulation of the numerical scheme, let $N \in \N$ be the number of  
temporal steps, let $k = \frac{T}{N}$ be the step size, and let
\begin{align}\label{eq:partition}
  \pi = \{0 = t_0 < \ldots <t_n< \ldots < t_N = T\}
\end{align}
be an equidistant partition of the interval $[0,T]$, where $t_n = n k$ for $n
\in \{0,\dots,N\}$. The backward Euler--Maruyama method for the Langevin
equation \eqref{eq:langevin} is then given by the recursion
\begin{align}
  \label{eq1:Euler}
  \begin{cases}
    X^{n} = X^{n-1} - k \nabla \Phi(X^{n}) + g_0 \Delta W^n, \quad 
    n \in \{1,\ldots,N\},\\
    X^0 = X_0,
  \end{cases}
\end{align}
where $\Delta W^n = W(t_n) - W(t_{n-1})$.

An example of a non-smooth potential is found by setting $d = m = 1$ and
$\Phi(x) = |x|^{p}$, $x \in \R$, for $p \in [1,2)$. Evidently, the gradient of
$\Phi$ is not locally Lipschitz continuous at $0 \in \R$ for $p \in (1,2)$.
Moreover, if $p = 1$, then the gradient $\nabla \Phi$ has a jump discontinuity
of the form 
\begin{align}
  \label{eq1:gradient}
  \nabla \Phi(x)
  =
  \begin{cases}
    -1,& \text{ if } x < 0,\\
    c,& \text{ if } x = 0,\\
    1,& \text{ if } x > 0.
  \end{cases}
\end{align}
Here, the value $c \in \R$ at $x =0$ is not canonically determined.
We have to solve a nonlinear equation of the form 
$x + k \nabla \Phi(x) = y$ in each step of the backward Euler method
\eqref{eq1:Euler}. However, if $y \in (-k,k)$, then the sole candidate for a
solution is $x = 0$, since otherwise $|x + k \nabla \Phi(x)| \ge k$. But $x=0$
is only a solution if $kc = y$. Therefore, the 
mapping $\R \ni x \mapsto x + k \nabla \Phi(x) \in \R$ is not surjective for
any single-valued choice of $c$.

This problem can be bypassed by considering the \emph{multi-valued
subdifferential} $\partial \Phi \colon \R^d \to 2^{\R^d}$ of a convex
potential $\Phi \colon \R^d \to \R$, which is given by 
\begin{align*}
  \partial \Phi (x) = \big\{ v \in \R^d\, : \, \Phi(x) + \langle v, y - x \rangle
  \le \Phi(y) \, \text{for all } y \in \R^d \big\}.
\end{align*}
Recall that $\partial \Phi(x) = \{\nabla \Phi(x)\}$ if the gradient exists at
$x \in \R^d$ in the classical sense. See
\cite[Section~23]{rockafellar1997} for further details.

In the above example, one easily verifies that 
\begin{align*}
  \partial \Phi(x) =
  \begin{cases}
    \{-1\},& \text{ if } x < 0,\\
    [-1,1],& \text{ if } x = 0,\\
    \{ 1\},& \text{ if } x > 0.
  \end{cases}
\end{align*}
This allows us to solve the nonlinear \emph{inclusion} where we want to 
find $x \in \R$ with $x + k \partial \Phi(x) \ni y$ for any $y \in \R$.

For this reason, we study the more general problem of the numerical
approximation of multi-valued stochastic differential equations (MSDE)
of the form
\begin{align}
  \label{eq1:SDE}
  \begin{cases}
    \diff{X(t)} + f(X(t)) \diff{t} \ni b(X(t)) \diff{t} 
    + g(X(t)) \diff{W(t)}, \quad t \in (0,T],\\
    X(0) = X_0.
  \end{cases}
\end{align}
Here, we assume that the mappings $b \colon \R^d \to \R^d$ and 
$g \colon \R^d \to \R^{d,m}$ are globally Lipschitz continuous. Moreover, the 
multi-valued drift coefficient function $f \colon \R^d \to 2^{\R^d}$ is assumed
to be a maximal monotone operator, cf.~Definition~\ref{def:mono} below.
See also Section~\ref{sec:exact} for a complete list of all imposed assumptions
on the MSDE \eqref{eq1:SDE}. Let us emphasize that the subdifferential of a
proper, lower semi-continuous and convex potential is an important example of a
possibly multi-valued and maximal monotone mapping $f$,
cf.~\cite[Corollary~31.5.2]{rockafellar1997}.

The backward Euler--Maruyama method for the approximation of 
the MSDE \eqref{eq1:SDE} on the partition $\pi$ is then given 
by the recursion
\begin{align}
  \label{eq1:multiEuler}
  \begin{cases}
    X^{n} \in X^{n-1} - k f(X^{n}) + k b(X^{n}) + g(X^{n-1}) \Delta W^n, \quad
    n \in \{1,\ldots,N\},\\
    X^0 = X_0.
  \end{cases}
\end{align}
We discuss the well-posedness of this method \eqref{eq1:multiEuler}
under our assumptions on $f$, $b$, and $g$ in Section~\ref{sec:wellposed}. In
particular, it will turn out that both problems, \eqref{eq1:SDE} and
\eqref{eq1:multiEuler}, admit single-valued solutions $(X(t))_{t \in [0,T]}$ and
$(X^n)_{n = 0}^N$, respectively.

The main result of this paper, Theorem~\ref{thm:conv}, then states that
the backward Euler--Maruyama method is convergent of order at least
$1/4$ with respect to the norm in $L^2(\Omega;\R^d)$. For the error
analysis we rely on techniques for deterministic problems developed in
\cite{nochetto2000}.  An important ingredient is the additional condition
on $f$ that there exists $\gamma \in (0,\infty)$ with
\begin{align*}
  \langle f_v - f_z, z - w \rangle \le \gamma \langle f_v - f_w, v - w\rangle 
\end{align*}
for all $v, w, z \in D(f) \subset \R^d$ and $f_v \in f(v)$, $f_w \in f(w)$,
$f_z \in f(z)$. This assumption is easily verified for a subdifferential of
a convex potential, cf.~Lemma~\ref{lem:NochettoIneq}. As already noted in
\cite{nochetto2000} for deterministic problems, this inequality allows us to
avoid Gronwall-type arguments in the error analysis for terms involving the
multi-valued mapping $f$.  

Before we give a more detailed outline of the content of this paper 
let us mention that multi-valued stochastic differential equations have been
studied in the literature before. The existence of a uniquely determined
solution to the MSDE \eqref{eq1:SDE} has been investigated, e.g., in 
\cite{Cepa1995, kree1982, pettersson1995}. We also refer to the more recent
monograph \cite{pardoux2014} and the references therein. In 
\cite{Gess2014, stephanPhd} related results have been derived for multi-valued
stochastic evolution equations in infinite dimension.
The numerical analysis for MSDEs has also been considered in
\cite{bernadin2003, lepingle2004, pettersson2000, wu2013, zhang2018}. 
However, these papers differ from the present paper in terms of
the considered numerical methods,
the imposed conditions, or the obtained order of convergence.

Further, we also mention that several authors have developed explicit numerical
methods for SDEs with discontinuous drifts in recent years. For instance, we
refer to \cite{dareiotis2018, Leobacher2017, Leobacher2018, mueller2018,
neuenkirch2019, NgoTaguchi2017, ngo2018}. While these results often apply to
more irregular drift coefficients, which are beyond the framework of maximal monotone
operators, the authors have to employ more restrictive conditions 
such as the global boundedness of the drift, which is not required in our
framework. 

This paper is organized as follows: In Section~\ref{sec:prelim} we fix some
notation and recall important terminology for multi-valued mappings. In
Section~\ref{sec:reg} we demonstrate how to apply
the techniques from \cite{nochetto2000} to the simplified setting
of the Langevin equation \eqref{eq:langevin}. In addition, we also show that
if the gradient $\nabla \Phi$ is more regular, say H\"older continuous with
exponent $\alpha \in (0,1]$, then the order of convergence increases to
$\frac{1 + \alpha}{4}$. Moreover, it turns out that 
the error constant does not grow exponentially with the final time $T$.
This is an important insight if the backward Euler method is used within an
unadjusted Langevin algorithm \cite{roberts1996}, which typically requires
large time intervals. See Theorem~\ref{thm:conv_Hoelder} and 
Remark~\ref{rem:errconst} below.

In Section~\ref{sec:exact} we turn to the more general multi-valued
stochastic differential equation \eqref{eq1:SDE}. We state all assumptions
imposed on the appearing drift and diffusion coefficients and collect some
properties of the exact solution. In Section~\ref{sec:wellposed} we show
that the backward Euler--Maruyama method \eqref{eq1:multiEuler} is well-posed
under the assumptions of Section~\ref{sec:exact}. In
Section~\ref{sec:aposteriori} we prove the already mentioned convergence
result with respect to the root-mean-square norm. Finally, in
Section~\ref{sec:examples} we verify that the setting of Section~\ref{sec:exact} 
applies to a Langevin equation with the discontinuous gradient
\eqref{eq1:gradient}. 
Further, we also show how to apply our results to the spatial discretization 
of the stochastic $p$-Laplace equation which indicates their usability for 
the numerical analysis of
stochastic partial differential equations.
However, a complete analysis of the latter problem will be deferred to a future
work. 

%
%

\section{Preliminaries}
\label{sec:prelim}

In this section, we collect some notation and introduce some background 
material.
First we recall some terminology for set valued mappings and (maximal)
monotone operators. For a more detailed introduction we refer, for 
instance, to \cite[Abschn.~3.3]{ruzicka2004} or 
\cite[Chapter~6]{Winkert2018}.

By $\R^d$, $d \in \N$, we denote the Euclidean space with the standard norm
$|\cdot|$ and inner product $\langle \cdot, \cdot \rangle$. 
Let $M \subset \R^d$ be a set. A set-valued mapping $f \colon M \to 
2^{\R^d}$ maps each $x \in M$ to an element of the power set $2^{\R^d}$, 
that is, $f(x) \subseteq \R^d$. The \emph{domain} $D(f)$ of $f$ is given by
\begin{align*}
  D(f) = \{ x \in M \, : \, f(x) \neq \emptyset \}.
\end{align*}

\begin{definition}
  \label{def:mono}
  Let $M \subset \R^d$ be a non-empty set. A set-valued mapping $f \colon M \to
  2^{\R^d}$ is called \emph{monotone} if 
  \begin{align*}
    \inner{f_u - f_v}{u - v} \ge 0
  \end{align*}
  for all $u,v \in D(f)$,  $f_u \in f(u)$, and $f_v \in f(v)$.

  Moreover, a set-valued mapping $f \colon M \to
  2^{\R^d}$ is called \emph{maximal monotone} if
  $f$ is monotone and for all $x \in M$ and $y \in \R^d$
  satisfying
  \begin{align*}
    \inner{y - f_v}{x - v} \ge 0 \quad \text{ for all } v \in D(f), f_v \in f(v),
  \end{align*}
it follows  that $x \in D(f)$ and $y \in f(x)$.
\end{definition}

Next, we recall a Burkholder--Davis--Gundy-type inequality.
For a proof we refer to \cite[Chapter 1, Theorem~7.1]{mao2008}.
For its formulation we take note that the Frobenius or Hilbert--Schmidt
norm of a matrix $g \in \R^{d,m}$ is also denoted by $|g|$. 

\begin{lemma}\label{lem:hoelder}
  Let $p \in [2,\infty)$ and $g \in L^p(\Omega; L^p(0,T;\R^{d,m}))$ be
  stochastically integrable. Then, for every $s,t \in [0,T]$ with $s<t$, the
  inequality 
  \begin{align*}
    \E \Big[ \Big| \int_s^t g(\tau) \diff{W(\tau)} \Big|^p \Big]
    &\leq \Big(\frac{p(p-1)}{2} \Big)^{\frac{p}{2}} (t-s)^{\frac{p-2}{2}}
    \E \Big[ \int_s^t |g(\tau)|^p \diff{\tau}  \Big]
  \end{align*}
  holds.
\end{lemma}



Let us also recall a stochastic variant of the Gronwall inequality.
A proof that can be modified to this setting can be found in \cite{xie2017}.
Compare also with \cite{scheutzow2013}.

\begin{lemma}\label{lem:StochGronwall}
  Let $Z,M,\xi \colon [0,T] \times \Omega \to \R$ be 
  $(\F_t)_{t\in[0,T]}$-adapted and almost surely continuous stochastic 
  processes such that $M$ is a local $(\F_t)_{t\in[0,T]}$-martingale with 
  $M(0) = 0$. Moreover, suppose that $Z$ and $\xi$ are nonnegative. In 
  addition, let $\varphi \colon [0, T ] \to \R$ be integrable and 
  nonnegative. If, for all $t \in [0, T ]$, we have
  \begin{align*}
    Z(t) \leq \xi(t) + \int_{0}^{t} \varphi(s) Z(s) \diff{s} + M(t), 
    \quad \P\text{-almost surely,}
  \end{align*}
  then, for every $t \in [0,T]$, the inequality 
  \begin{align*}
    \E \big[Z(t) \big] \leq \exp\Big(\int_{0}^{t} \varphi(s) \diff{s} \Big) 
    \E\big[ \sup_{s\in[0,t]} \xi(s) \big]
  \end{align*}
  holds.
\end{lemma}

Moreover, we often make use of generic constants. More precisely, by $C$ we
denote a finite and positive quantity that may vary from occurrence to
occurrence but is always independent of numerical parameters such as the step
size $k = \frac{T}{N}$ and the number of steps $N \in \N$.

%
%

\section{Application to the Langevin equation with a convex potential}
\label{sec:reg}

In order to illustrate our approach, we first consider a more regular
stochastic differential equation with single-valued (H\"older)
continuous drift term.  More precisely, we consider the overdamped
Langevin equation \cite[Section~2.2]{lelievre2010}
\begin{align}
  \label{eq2:SDE}
  \begin{cases}
    \diff{X(t)} = - \nabla \Phi(X(t))\diff{t} + g_0 \diff{W(t)}, \quad t \in
    [0,T],\\
    X(0) = X_0,
  \end{cases}
\end{align}
where $X_0 \in L^2(\Omega, \F_0, \P;\R^d)$, $g_0 \in \R^{d,m}$, and 
$W \colon [0,T] \times \Omega \to \R^m$ is a standard $\R^m$-valued 
Wiener process. In addition, we impose the following assumption on the
potential $\Phi \colon \R^d \to \R$.

\begin{assumption}
  \label{ass:phi}
  Let $\Phi \colon \R^d \to \R$ be a convex, nonnegative, and
  continuously differentiable function.
\end{assumption}

In the following, we denote by $f \colon \R^d \to \R^d$ the gradient of
$\Phi$, that is $f(x) = \nabla \Phi(x)$. It is well-known that 
the convexity of $\Phi$ implies the variational inequality 
\begin{align}
  \label{eq1:varineq}
  \langle f(v), w - v \rangle \le \Phi(w) - \Phi(v),\quad v, w \in \R^d,
\end{align}
see, for example, 
\cite[\S~23]{rockafellar1997}.

In the following lemma, we collect some properties of $f$, which are
direct consequences of Assumption~\ref{ass:phi}. Both inequalities are
well-known.  The proof of \eqref{eq4:NochettoIneq} is taken from
\cite{nochetto2000}.

\begin{lemma} 
  \label{lem:NochettoIneq}
  Under Assumption~\ref{ass:phi} and with $f = \nabla \Phi$,  the 
  inequalities
  \begin{align}
    \label{eq4:monotonicity}
    \inner{f(v) - f(w)}{v-w} \ge 0
  \end{align}
  and 
  \begin{align}
    \label{eq4:NochettoIneq}
    \inner{f(v) - f(z)}{z-w}
    \leq \inner{f(v) - f(w)}{v-w}
  \end{align}
  are fulfilled for all $v,w,z \in \R^d$. 
\end{lemma}

\begin{proof}
  The first inequality follows directly from \eqref{eq1:varineq} since
  \begin{align*}
    \inner{f(v) - f(w)}{v-w} &=  -\inner{f(v)}{w-v} -  \inner{f(w)}{v-w} \\
    &\geq - \big( \Phi(w) - \Phi(v) \big) - \big( \Phi(v) - \Phi(w) \big) = 0
  \end{align*}
  for all $v,w \in \R^d$.
  For the proof of the second inequality we start by rewriting its left-hand
  side. For arbitrary $v,w,z \in \R^d$ we
  rearrange the terms to obtain
  \begin{align*}
    &\inner{f(v) - f(z)}{z-w}\\
    &\quad = \inner{f(v)}{z} - \inner{f(v)}{w} + \inner{f(z)}{w-z}
      \pm \inner{f(v)}{v}\\
    &\quad = \inner{f(v)}{z-v} + \inner{f(v)}{v-w} + \inner{f(z)}{w-z}
    \pm \inner{f(w)}{v-w}\\
    &\quad = \inner{f(v)}{z-v} + \inner{f(v)-f(w)}{v-w} + \inner{f(z)}{w-z}
    + \inner{f(w)}{v-w}.
  \end{align*}
  Setting $\sigma(v,w) := \Phi(w) - \Phi(v) - \inner{f(v)}{w-v}$ for all $v,w
  \in \R^d$, we see that
  \begin{align*}
    \inner{f(v) - f(z)}{z-w}
    &= \inner{f(v)-f(w)}{v-w} 
    + \Phi(z) - \Phi(v) - \sigma(v,z) \\
    &\quad + \Phi(w) - \Phi(z) - \sigma(z,w)
    + \Phi(v) - \Phi(w) - \sigma(w,v)\\
    &= \inner{f(v)-f(w)}{v-w} - \sigma(v,z) - \sigma(z,w) - \sigma(w,v).
  \end{align*}
  But \eqref{eq1:varineq} says that $\sigma(v,w) \geq 0$ for all
  $v,w \in \R^d$, which completes the proof. 
\end{proof}

It follows from Assumption~\ref{ass:phi} and
Lemma~\ref{lem:NochettoIneq} that the drift $f = \nabla \Phi$ of the
stochastic differential equation \eqref{eq2:SDE} is continuous and 
monotone.
Therefore, by \cite[Thm.~3.1.1]{roeckner2007} the stochastic differential
equation \eqref{eq2:SDE} has a solution in the strong (probabilistic) sense
satisfying $\P$-a.s.\ for all $t \in [0,\infty)$
\begin{align}
  \label{eq2:SDEsol}
  X(t) = X_0 - \int_0^t f(X(s))\diff{s} + g_0 W(t).
\end{align}
Moreover, the solution is unique up to $\P$-indistinguishability and it is
square-integrable with
\begin{align*}
  \sup_{t \in [0,T]} \E \big[ |X(t)|^2 \big] 
  \le C \big(1 + \E \big[ |X_0|^2 \big] \big).
\end{align*}

Next, we turn to the numerical approximation of the solution of
\eqref{eq2:SDE}. Recall that for a single-valued drift 
the backward Euler--Maruyama method is given by the 
recursion
\begin{align}
  \label{eq2:bEuler}
  \begin{cases}
    X^{n} = X^{n-1} - k f(X^{n}) + g_0 \Delta W^n, \quad n \in
    \{1,\ldots,N\},\\
    X^0 = X_0,
  \end{cases}
\end{align}
where $\Delta W^n = W(t_{n}) - W(t_{n-1})$, $t_n = nk$ and $k = 
\frac{T}{N}$.

The next lemma contains some \emph{a priori} estimates for the backward
Euler--Maruyama method \eqref{eq2:bEuler}.

\begin{lemma}
  \label{lem:apriori2}
  Let $g_0 \in \R^{d,m}$ be given and let Assumption~\ref{ass:phi} be 
  satisfied.
  For an arbitrary step size $k = \frac{T}{N}$, $N \in \N$, 
  let $(X^n)_{n \in \{0,\dots,N\}}$ be a family of
  $(\F_{t_n})_{n \in \{0,\ldots,N\}}$-adapted random
  variables satisfying \eqref{eq2:bEuler}. If $X_0 \in
  L^2(\Omega,\F_0,\P;\R^d)$, then 
  \begin{align}
    \label{eq2:aprioriX1}
    \max_{n \in \{1,\ldots,N\}} \E [|X^n|^2]
    \leq \E \big[ |X_0|^2 \big] + 2 T \big( \Phi(0) + |g_0|^2 \big)
  \end{align}
  and 
  \begin{align}
    \label{eq2:aprioriX}
    \sum_{n = 1}^N \E \big[ | X^n - X^{n-1}|^2 \big]
    + 4 k \sum_{n = 1}^{N} \E \big[ \Phi(X^n) \big]
    \leq 2\E \big[ |X_0|^2 \big] + 4 T \big( \Phi(0) + |g_0|^2 \big).
  \end{align}
\end{lemma}

\begin{proof}
  First, we recall the identity
  \begin{align*}
    \inner{X^n - X^{n-1}}{X^n}
    = \frac{1}{2} \big( |X^n|^2 - |X^{n-1}|^2 + |X^n - X^{n-1}|^2 \big).
  \end{align*}
  Using also \eqref{eq2:bEuler}, we then get
  \begin{align*}
    |X^n|^2 - |X^{n-1}|^2 + |X^n - X^{n-1}|^2
    &= 2 \langle  X^n - X^{n-1}, X^n \rangle \\
    &= - 2 k \inner{f(X^n)}{X^n}
    + 2 \langle g_0 \Delta W^n, X^n \rangle,
  \end{align*}
  for every $n \in \{1,\ldots,N\}$.
  Hence, an application of \eqref{eq1:varineq} yields
  \begin{align*}
    |X^n|^2 - |X^{n-1}|^2 + |X^n - X^{n-1}|^2
    &\le 2 k \big( \Phi(0) - \Phi(X^n)  \big)
    + 2 \langle g_0 \Delta W^n, X^n \rangle,
  \end{align*}
  for every $n \in \{1,\ldots,N\}$.
  From applications of the Cauchy--Schwarz inequality 
  and the weighted Young inequality
  we then obtain 
  \begin{align*}
    &|X^n|^2 - |X^{n-1}|^2 + |X^n - X^{n-1}|^2 + 2 k \Phi(X^n) \\
    &\quad \le 2 k \Phi(0) + 2 \langle g_0 \Delta W^n, X^n - X^{n-1}
    \rangle + 2 \langle g_0 \Delta W^n, X^{n-1} \rangle\\
    &\quad \le 2 k \Phi(0) + 2 \big| g_0 \Delta W^n \big|^2
    + \frac{1}{2} | X^n - X^{n-1}|^2
    + 2 \langle g_0 \Delta W^n, X^{n-1} \rangle,
  \end{align*}
  for every $n \in \{1,\ldots,N\}$.
  
  The third term on the right-hand side is absorbed in the third term
  on the left-hand side. Summation then yields 
  \begin{align*}
    &|X^n|^2 + \frac{1}{2} \sum_{ j = 1}^n | X^{j} - X^{j-1} |^2
    + 2 k \sum_{j = 1}^n \Phi(X^j)\\
    &\quad \le |X^0|^2 + 2 t_n \Phi(0) 
    + 2 \sum_{j = 1}^n \big| g_0 \Delta W^j \big|^2
    + 2 \sum_{j = 1}^n \big\langle g_0 \Delta W^j, X^{j-1} \big\rangle.
  \end{align*}
  An inductive argument over $n \in \{1,\ldots,N\}$ then yields that $X^n$ is
  square-integrable due to the assumption $X_0 \in L^2(\Omega,\F_0,\P;\R^d)$.
  Therefore, after taking expectation the last sum vanishes. Moreover, an
  application of the It\=o isometry then gives
  \begin{align*}
    &\E\big[ |X^n|^2 \big] + \frac{1}{2} \sum_{ j = 1}^n \E \big[ | X^{j} -
    X^{j-1} |^2 \big]
    + 2 k \sum_{j = 1}^n \E \big[ \Phi(X^j) \big]\\
    &\quad \le \E \big[ |X^0|^2 \big] + 2 t_n \Phi(0)
    + 2 \sum_{j = 1}^n \E \big[ \big| g_0 \Delta W^j \big|^2 \big]\\
    &\quad = \E \big[ |X_0|^2 \big] + 2 t_n \big( \Phi(0) +
    | g_0 |^2\big).
  \end{align*}
  Since this is true for any $n \in \{1,\ldots,N\}$ the assertion follows.
\end{proof}

As the next theorem shows, Assumption~\ref{ass:phi} is also sufficient
to ensure the well-posedness of the backward Euler--Maruyama method.
The result follows directly from the fact that $f$ is continuous and 
monotone
due to \eqref{eq4:monotonicity}. For a proof we refer, for instance,
to \cite[Sect.~4]{beyn2016}, \cite[Chap.~6.4]{ortega2000}, and
\cite[Theorem~C.2]{stuart1996}. The assertion also follows from the more
general result in Theorem~\ref{thm:existence_multi} below.

\begin{theorem}
  \label{thm:existence}
  Let $X_0 \in L^2(\Omega,\F_0,\P;\R^d)$ as well as  $g_0 \in \R^{d,m}$ be 
  given and let Assumption~\ref{ass:phi} be satisfied.
  Then, for every equidistant step size $k = \frac{T}{N}$, $N \in \N$,
  there exists a uniquely determined family of square-integrable and
  $(\F_{t_n})_{n \in \{0,\ldots,N\}}$-adapted random variables $(X^n)_{n\in
  \{0,\dots,N\}}$ satisfying \eqref{eq2:bEuler}.
\end{theorem}

We now turn to an error estimate with respect to the 
$L^2(\Omega;\R^d)$-norm. 
Since we do not impose any (local) Lipschitz condition on the 
drift $f$, 
classical approaches based on discrete Gronwall-type inequalities
are not applicable.
Instead we rely on an error representation formula, which was introduced  
for deterministic problems in \cite{nochetto2000}.

For the formulation of this, we introduce the following notation:
For a given equidistant partition $\pi = \{0=t_0 < t_1 < \ldots < t_N = T\}
\subset [0,T]$ with step size $k = \frac{T}{N}$, we denote by $\X \colon [0,T]
\times \Omega \to \R^d$ the piecewise linear interpolant of the
sequence $(X^n)_{n \in \{0,\ldots,N\}}$ generated by the backward Euler method
\eqref{eq2:bEuler}.   It is defined by $\X(0) = X^0$ and 
\begin{align}
  \label{eq2:Xlin}
  \X(t) = \frac{t - t_{n-1}}{k} X^n + \frac{t_n - t}{k} X^{n-1}, \quad 
  \text{for all } t \in (t_{n-1}, t_n], \, n \in \{1,\ldots,N\}.
\end{align}
In addition, we introduce the processes 
$\overline{\X}, \underline{\X} \colon [0,T] \times \Omega \to \R^d$,
which are piecewise constant interpolants of $(X^n)_{n \in \{0,\ldots,N\}}$
and defined by $\overline{\X}(0)= \underline{\X}(0) = X^0$ and
\begin{align}
  \label{eq2:Xconst}
  \overline{\X}(t) = X^n \quad \text{and} \quad \underline{\X}(t) =
  X^{n-1}, \quad \text{ for all } t \in (t_{n-1},t_n], \, n \in \{1,\ldots,N\}.
\end{align}
Analogously, we define the piecewise linear interpolated process
$\mathcal{W} \colon [0,T] \times \Omega \to \R^m$ by $\mathcal{W}(0) = 0$ 
and
\begin{align}
  \label{eq2:Wlin}
  \mathcal{W}(t) = \frac{t - t_{n-1}}{k} W(t_{n}) + \frac{t_n - t}{k}
  W(t_{n-1}) = W(t_{n-1}) + \frac{t - t_{n-1}}{k} \Delta W^n, 
\end{align}
for all $t \in (t_{n-1},t_n]$, $n \in \{1,\ldots,N\}$.

We are now prepared to state Lemma~\ref{lem:err_rep2}.
The underlying idea of this lemma was introduced in 
\cite{nochetto2000}, where it is used to derive a posteriori error estimates 
for the backward Euler method. In fact, in the absence of noise,
only the first term on the right-hand side of \eqref{eq2:errorRep} is
non-zero. In \cite{nochetto2000} this term is used as an \emph{a posteriori
error estimator}, since it is explicitly computable by quantities 
generated by the numerical method. 

\begin{lemma}
  \label{lem:err_rep2}
  Let $X_0 \in L^2(\Omega,\F_0,\P;\R^d)$ as well as  $g_0 \in \R^{d,m}$ be 
  given and let Assumption~\ref{ass:phi} be satisfied. Let $k = \frac{T}{N}$, 
  $N \in \N$, 
  be an arbitrary equidistant step size and let $t_n = nk$, $n \in \{0,\dots, 
  N\}$. Then, for every $n \in \{1,\ldots,N\}$
   the estimate
  \begin{align}\label{eq2:errorRep}
    \begin{split}
      \E \big[ | X(t_n) - X^n|^2 \big] 
      &\le  k \sum_{i = 1}^n \E \big[ \inner{f(X^i) - f(X^{i-1})}{X^i -
      X^{i-1}}\big]\\
      &\quad + 2 \int_0^{t_n} \E \big[ \big\langle f( \overline{\X}(t)) -
      f(X(t)), g_0 \big( \mathcal{W}(t) - W(t)\big) \big\rangle \big] \diff{t}
    \end{split}
  \end{align}
  holds, where $(X(t))_{t \in [0,T]}$ and $(X^n)_{n \in \{0,\ldots,N\}}$ are the 
  solutions of 
  \eqref{eq2:SDE} and \eqref{eq2:bEuler}, respectively.
\end{lemma}

\begin{proof}
  From \eqref{eq2:bEuler} we directly deduce that for every $n \in
  \{1,\ldots,N\}$ 
  \begin{align*}
    X^n = X_0 - k \sum_{i = 1}^n f(X^{i}) + g_0 W(t_n).
  \end{align*}
  Then, one easily verifies for all $t \in (t_{n-1},t_n]$, $n \in
  \{1,\ldots,N\}$, that
  \begin{align*}
    \X(t) = X_0 - \int_0^t f(\overline{\X}(s)) \diff{s} + g_0 \mathcal{W}(t).
  \end{align*}
  Hence, due to \eqref{eq2:SDEsol} the error process $E := X - \X$ fulfills  
  \begin{align}
    \label{eq2:errorX}
    E(t) = \int_0^t f(\overline{\X}(s)) - f(X(s)) \diff{s} 
    + g_0 \big(W(t) - \mathcal{W}(t) \big) =: E_1(t) + E_2(t)
  \end{align}
  for all $t \in [0,T]$.
  Here, we have $E_2(t_n)= 0$, since
  $\mathcal{W}$ is an interpolant of $W$. Hence, 
  for all $n \in \{0,\ldots,N\}$, 
  \begin{align}
    \label{eq2:normError}
    |E(t_n)|^2 =  | E_1(t_n)|^2.
  \end{align}
  To estimate the norm of
  $E_1(t_n)$, we first note that $E_1$ has absolutely continuous
  sample paths with $E_1(0)=0$. Hence, 
  \begin{align*}
    \frac{1}{2} \frac{\mathrm{d}}{\diff{t}} |E_1(t)|^2
    =  \inner{\dot{E}_1(t)}{E_1(t)}
  \end{align*}
  is fulfilled for almost all $t \in [0,T]$.
  Therefore, by integration with respect to $t$, we get
  \begin{align}\label{eq:E1}
    \begin{split}
    \frac{1}{2} |E_1(t_n)|^2 &= \int_0^{t_n}
    \inner{\dot{E}_1(t)}{E_1(t)}\diff{t}
    \\ & = \int_0^{t_n}
    \inner{\dot{E}_1(t)}{E(t)}\diff{t}
    - \int_0^{t_n}
    \inner{\dot{E}_1(t)}{E_2(t)}\diff{t}.
    \end{split}
  \end{align}
  Next, we write
  \begin{align*}
    \X(t) 
    =  \frac{t - t_{n-1}}{k} \overline{\X}(t) + \frac{t_n - t}{k}
    \underline{\X}(t),
    \quad t \in (t_{n-1}, t_n], 
  \end{align*}
  and 
  use \eqref{eq4:monotonicity} and \eqref{eq4:NochettoIneq}
  to obtain, for almost every 
  $t \in (t_{n-1}, t_n]$, that
  \begin{align*}
    \inner{\dot{E}_1(t)}{E(t)} 
    &= \inner{f(\overline{\X}(t)) - f(X(t))}{X(t) - \X(t)}\\
    &=
      \frac{t - t_{n-1}}{k} 
    \inner{f(\overline{\X}(t)) - f(X(t))}{X(t) - \overline{\X}(t)}\\
    &\quad
      +
      \frac{t_n -t}{k}
      \inner{f(\overline{\X}(t)) - f(X(t))}{X(t) - \underline{\X}(t)}\\
    &\le \frac{t_n -t}{k} \inner{f(\overline{\X}(t)) -
    f(\underline{\X}(t))}{\overline{\X}(t) - \underline{\X}(t)}\\
    &= \frac{t_n -t}{k} \inner{f(X^n) -
    f(X^{n-1})}{X^n - X^{n-1}}.
  \end{align*}
  Furthermore, the expectation of the second integral on the right-hand side 
  of \eqref{eq:E1} is equal to
  \begin{align*}
    \E \Big[ \int_0^{t_n} \inner{\dot{E}_1(t)}{E_2(t)}\diff{t} \Big]
    &= \int_0^{t_n} \E \big[ \inner{f(\overline{\X}(t)) - f(X(t))}{g_0(W(t) -
    \mathcal{W}(t))} \big] \diff{t}.
  \end{align*}
  Therefore,
    \begin{align*}
    \E \big[ |E_1(t_n)|^2 \big] 
    &= 2 \int_0^{t_n} \E [ \inner{\dot{E}_1(t)}{E(t)} ] \diff{t}
    - 2 \int_0^{t_n} \E[ \inner{\dot{E}_1(t)}{E_2(t)} ] \diff{t}\\
    &\le 2 \sum_{i = 1}^n \int_{t_{i-1}}^{t_i} \frac{t_i -t}{k} \diff{t}
    \, \E\big[ \inner{f(X^i) - f(X^{i-1})}{X^i - X^{i-1}} \big]\\
    &\quad + 2 \int_0^{t_n} \E \big[ \inner{f(\overline{\X}(t)) -
    f(X(t))}{g_0(\mathcal{W}(t) - W(t))} \big] \diff{t}.
  \end{align*}
  Since $\int_{t_{i-1}}^{t_i} (t_i - t) \diff{t} = \frac{1}{2} k^2$ the assertion 
  follows. 
\end{proof}

The next lemma contains an estimate of the difference between the Wiener
process $W$ and its piecewise linear interpolant $\mathcal{W}$.

\begin{lemma}
  \label{lem:Werror}
  For every $g_0 \in \R^{d,m}$ and every step size $k = \frac{T}{N}$, $N \in
  \N$, the equality
  \begin{align}
    \label{eq5:errGL2const}
    \Big( \int_{0}^{T} \E [ |g_0 ( W(t) - \mathcal{W}(t)) |^2] \diff{t}
    \Big)^{\frac{1}{2}}
    = \frac{1}{\sqrt{6}} T^{\frac{1}{2}} | g_0 | k^{\frac{1}{2}} 
  \end{align}
  holds.
\end{lemma}

\begin{proof}
  From the definition \eqref{eq2:Wlin} of $\mathcal{W}$ it follows that
  \begin{align*}
    &\int_{0}^{T} \E [ |g_0 ( W(t) - \mathcal{W}(t)) |^2] \diff{t}\\
    &\quad = \sum_{n = 1}^N \int_{t_{n-1}}^{t_n} \E \Big[ \Big|g_0 \Big( W(t)
    - W(t_{n-1}) - \frac{t-t_{n-1}}{k} \Delta W^n \Big) \Big|^2 \Big] \diff{t}
    \\ 
    &\quad =  \sum_{n =1}^N \int_{t_{n-1}}^{t_n}
    \E \Big[ \Big| \frac{t_n - t}{k}  g_0 ( W(t) -
    W(t_{n-1})) - \frac{t-t_{n-1}}{k} g_0 (W(t_n) - W(t)) \Big|^2 \Big]
    \diff{t}\\
    &\quad = \sum_{n =1}^N \int_{t_{n-1}}^{t_n} \E \Big[ \Big| \frac{t_n -
    t}{k}  g_0 ( W(t) - W(t_{n-1})) \Big|^2 +
     \Big| \frac{t-t_{n-1}}{k} g_0 (W(t_n) -
    W(t))\Big|^2 \Big] \diff{t}\\
    &\quad = \frac{1}{k^2}\sum_{n =1}^N \Big(\int_{t_{n-1}}^{t_n} |g_0|^2 (t_n -
    t)^2 (t- t_{n-1}) \diff{t} +
    \int_{t_{n-1}}^{t_n} |g_0|^2 (t - t_{n-1})^2 (t_n- t) \diff{t}\Big),
  \end{align*}
  where we used that the two increments of the Wiener process are 
  independent for every $t \in (t_{n-1},t_n]$, $n \in \{1,\ldots, N\}$, and  we 
  also
  applied It\=o's isometry.
  By symmetry of the two terms it then follows that
  \begin{align*}
    \int_{0}^{T} \E [ |g_0 ( W(t) - \mathcal{W}(t)) |^2] \diff{t}
    &= \frac{1}{6} T |g_0|^2 k,
  \end{align*}
  and the proof is complete.
\end{proof}

The error estimates in Lemma~\ref{lem:err_rep2} and Lemma~\ref{lem:Werror}
allow us to determine the order of convergence of the backward Euler--Maruyama
method without relying on discrete Gronwall-type inequalities.
The following theorem imposes the additional assumption that the drift $f$ 
is H\"older continuous. We include the parameter value $\alpha = 0$, which simply
means that $f$ is continuous and globally bounded. The case of less regular $f$
is treated in Section~\ref{sec:aposteriori}.

Observe that we recover the standard rate $\frac{1}{2}$ if $\alpha = 1$, that
is, if the drift $f$ is assumed to be globally Lipschitz continuous. 
Compare also with the standard literature, for example,
\cite[Chap.~12]{kloeden1992} or \cite[Sect.~1.3]{milstein2004}. 

For processes $X \colon [0,T] \times \Omega \to \R^d$ and exponents
$\alpha\in [0,1]$, we define the family of H\"older semi-norms by
\begin{align*}
  C^{\alpha}([0,T];L^2(\Omega;\R^d))
  =\sup_{\stackrel{t,s\in [0,T]}{t\neq s}}
  \frac{\|X(t)-X(s)\|_{L^2(\Omega;\R^d)}}{|t-s|^{\alpha}}.  
\end{align*}

\begin{theorem}
  \label{thm:conv_Hoelder}
  Let $X_0 \in L^2(\Omega,\F_0,\P;\R^d)$ as well as  $g_0 \in \R^{d,m}$ be 
  given, let Assumption~\ref{ass:phi} be fulfilled and let $f = \nabla \Phi$
  be H\"older continuous with exponent $\alpha \in [0,1]$, 
  i.e., there exists $L_f \in ( 0,\infty)$ such that
  \begin{align*}
    | f(x) - f(y) | \leq L_f |x-y|^\alpha, \quad \text{ for all } x,y \in \R^d.
  \end{align*}
  Then there exists $C \in (0,\infty)$ such that for every step size $k =
  \frac{T}{N}$, $N \in \N$, the estimate 
  \begin{align*}
    \max_{n \in \{0,\ldots,N\}} \| X(t_n) - X^n \|_{L^2(\Omega;\R^d)}
    \le C k^{\frac{1 + \alpha}{4}}
  \end{align*}
holds,  where $(X(t))_{t \in [0,T]}$ and $(X^n)_{n \in \{0,\ldots,N\}}$ are the
  solutions to \eqref{eq2:SDE} and \eqref{eq2:bEuler}, respectively.
\end{theorem}

\begin{proof}
  Since $f$ is assumed to be $\alpha$-H\"older continuous it follows that
  \begin{align*}
    |f(x)| \le \max(L_f,|f(0)|)( 1 + |x|^\alpha), \quad \text{for all $x \in \R^d$.}
  \end{align*}
  In particular, $f$ grows at most linearly. Therefore,
  as stated in \cite[Chap.~2, Thm~4.3]{mao2008}, the solution 
  $(X(t))_{t\in [0,T]}$ of \eqref{eq2:SDE} satisfies
  $X \in C^{\frac{1}{2}}([0,T];L^2(\Omega;\R^d))$.

  We will use Lemma~\ref{lem:err_rep2} to prove the error bound. To this 
  end, we first show that
  \begin{align}
  \label{eq5:H1}
    \begin{split}
      &k \sum_{i=1}^{N} \E \big[ \inner{f(X^i) - f(X^{i-1})}{X^i -
        X^{i-1}}\big]\\
      &\quad \leq L_f T^{\frac{1-\alpha}{2}}
      \Big( 2 \E \big[ |X_0|^2 \big] + 4 T \big( \Phi(0) + |g_0|^2 \big) 
      \Big)^{\frac{1+\alpha}{2}} k^{\frac{1+\alpha}{2}}.
    \end{split}
  \end{align}
  Indeed, we make use of the H\"older continuity of $f$ directly and 
  obtain
  \begin{align*}
    &k \sum_{i=1}^{N} \E \big[ \inner{f(X^i) - f(X^{i-1})}{X^i -
      X^{i-1}}\big]\\ 
    &\quad \leq \sum_{i=1}^{N} k \E \big[ |f(X^i) - f(X^{i-1})| |X^i -
    X^{i-1}|\big]\\ 
    &\quad \leq  L_f \sum_{i=1}^{N} k^{\frac{1}{q}} k^{\frac{1}{p}}
    \E \big[ |X^i - X^{i-1}|^{1 + \alpha}
    \big]\\
    &\quad \leq L_f \Big( \sum_{i = 1}^N k \Big)^{\frac{1}{q}} 
    \Big( k \sum_{i = 1}^N \E \big[ |X^i - X^{i-1}|^2 \big]
    \Big)^{\frac{1}{p}}, 
  \end{align*}
  where we also used H\"older's inequality with $p = \frac{2}{1 + \alpha} \in [1,2]$ and
  $\frac{1}{p} + \frac{1}{q} = 1$ as well as Jensen's inequality. Due
  to the \emph{a priori}
  estimate \eqref{eq2:aprioriX} the sum $\sum_{i=1}^{N} \E \big[ |X^i -
  X^{i-1}|^2 \big]$ is bounded independently of the step size $k$. Hence, we
  arrive at \eqref{eq5:H1}. 
  
  Therefore, it remains to estimate the second error term in Lemma~\ref{lem:err_rep2}:
    \begin{align}
    \label{eq2:term2}
    \begin{split}
      &\int_0^{t_n} \E \big[
      \big\langle f(\overline{\X}(t)) - f(X(t)), g_0 
    (\mathcal{W}(t)
    - W(t) ) \big\rangle \big] \diff{t}\\
    &\quad = \sum_{j = 1}^n \int_{t_{j-1}}^{t_j}
    \E \big[ \big\langle f(X^j) - f(X(t)), g_0 (\mathcal{W}(t)
    - W(t) ) \big\rangle \big] \diff{t},
    \end{split}
  \end{align}
  where we inserted the definition of $\overline{\X}$ from
  \eqref{eq2:Xconst}. 
  Moreover, from \eqref{eq2:Wlin} we get
  \begin{align*}
    g_0 (\mathcal{W}(t) - W(t) ) = \frac{t - t_{j-1}}{k} g_0 \Delta
    W^j - g_0 ( W(t) - W(t_{j-1}) ) 
  \end{align*}
  for $t \in (t_{j-1},t_j]$.  Hence, the random variable in the second
  slot of the inner product on the right-hand side \eqref{eq2:term2}
  is centered and is independent of any $\F_{t_{j-1}}$-measurable 
  random variable. Thus, we may write
  \begin{align*}
    &\sum_{j = 1}^n \int_{t_{j-1}}^{t_j}
    \E \big[ \big\langle f(X^j) - f(X(t)), g_0 (\mathcal{W}(t) - W(t) )
    \big\rangle \big] \diff{t}\\
    &\quad = \sum_{j = 1}^n \int_{t_{j-1}}^{t_j}
    \E \big[ \big\langle f(X^j) - f(X^{j-1}),g_0 (\mathcal{W}(t) - W(t) )
    \big\rangle \big] \diff{t}\\
    &\qquad + \sum_{j = 1}^n \int_{t_{j-1}}^{t_j}
    \E \big[ \big\langle f(X(t_{j-1})) - f(X(t)), 
    g_0 (\mathcal{W}(t) - W(t) ) \big\rangle \big] \diff{t} =: T_1 + T_2.
  \end{align*}
  To estimate $T_1$ we first recall the definitions of $\underline{\X}$ and
  $\overline{\X}$ from \eqref{eq2:Xconst}. Then we
  apply the Cauchy--Schwarz inequality and obtain
  \begin{align*}
    T_1 &= \int_0^{t_n} \E \big[ \langle f(\overline{\X}(t)) -
    f(\underline{\X}(t)), g_0 (\mathcal{W}(t)
    - W(t) ) \rangle \big] \diff{t}\\
    &\le \Big( \int_0^{t_n} \E \big[  | f(\overline{\X}(t)) -
    f(\underline{\X}(t)) |^2 \big] \diff{t} \Big)^{\frac{1}{2}}
    \Big( \int_0^{t_n} \E \big[  | g_0 (\mathcal{W}(t)
    - W(t) ) |^2 \big] \diff{t} \Big)^{\frac{1}{2}}.
  \end{align*}
  From the H\"older continuity of $f$ we then deduce that 
  \begin{align*}
    \int_0^{t_n} \E \big[  | f(\overline{\X}(t)) -
    f(\underline{\X}(t)) |^2 \big] \diff{t}
    &\le L_f^2 k \sum_{i = 1}^N \E \big[  | X^{i} - X^{i-1}|^{2 \alpha} \big]\\
    &\le L_f^2 T^{\frac{1}{q}}
    \Big(k \sum_{i =1}^N \E \big[  | X^{i} - X^{i-1}|^{2} \big]
    \Big)^{\alpha},
  \end{align*}
  where the last inequality is in fact an equality if $\alpha = 1$,
  $\frac{1}{q} = 0$ or if $\alpha = 0$, $\frac{1}{q}=1$. Otherwise
  the inequality follows from  H\"older's
  inequality with $p = \frac{1}{\alpha} \in (1,\infty)$ and
  $\frac{1}{p} + \frac{1}{q} = 1$, followed by an application of Jensen's
  inequality. Furthermore, Lemma~\ref{lem:Werror} states that
  \begin{align}
    \label{eq2:Wint}
    \Big( \int_0^T \E \big[ | g_0 (\mathcal{W}(t)
    - W(t) ) |^2 \big] \diff{t} \Big)^{\frac{1}{2}}
    = \frac{1}{\sqrt{6}} T^{\frac{1}{2}} |g_0| k^{\frac{1}{2}}.
  \end{align}
  Therefore, together with \eqref{eq2:aprioriX} we arrive at the estimate 
  \begin{align*}
    T_1 \le \frac{1}{\sqrt{6}} L_f T^{\frac{2 - \alpha}{2}} |g_0| \Big( 
    2 \E \big[ |X_0|^2 \big] + 4 T \big( \Phi(0) + |g_0|^2 \big)
    \Big)^{\frac{\alpha}{2}} k^{\frac{1 + \alpha}{2}}
  \end{align*}
  for all $n \in \{1,\ldots,N\}$.
  
  The estimate of $T_2$ works similarly by additionally making use of the
  H\"older continuity of the exact solution. To be more precise, we have 
  that 
  \begin{align*}
    \sum_{i =1}^n \int_{t_{i-1}}^{t_i} \E \big[  | f(X(t_{i-1})) -
    f(X(t)) |^2 \big] \diff{t}
    &\le L_f^2 \sum_{i = 1}^N \int_{t_{i-1}}^{t_i} 
    \E \big[  | X(t_{i-1}) - X(t) |^{2 \alpha} \big] \diff{t}\\
    &\le L_f^2 \sum_{i = 1}^N \int_{t_{i-1}}^{t_i} 
    \big( \E \big[  | X(t_{i-1}) - X(t) |^{2} \big] \big)^{\alpha} \diff{t}\\
    &\le L_f^2 T \|X \|_{C^{\frac{1}{2}}([0,T];L^2(\Omega;\R^{d}))}^{2\alpha}
    k^{\alpha}.    
  \end{align*}
  Together with the Cauchy--Schwarz inequality and 
  \eqref{eq2:Wint}, we therefore obtain
  \begin{align*}
    T_2 \le \frac{1}{\sqrt{6}} L_f T |g_0|
    \|X \|_{C^{\frac{1}{2}}([0,T];L^2(\Omega;\R^{d}))}^\alpha k^{\frac{1 + 
    \alpha}{2}}.
  \end{align*}
  Inserting the estimates for $T_1$, $T_2$ and \eqref{eq5:H1}
  into Lemma~\ref{lem:err_rep2} completes the proof.
\end{proof}

\begin{remark}
  \label{rem:errconst}
  The precise form of the constant $C$ appearing in
  Theorem~\ref{thm:conv_Hoelder} is, after taking squares,  
  \begin{align*}
    C^2 &= L_f T^{\frac{1-\alpha}{2}} C_0^{\frac{1+\alpha}{2}} 
    + \frac{1}{\sqrt{6}} L_f T^{\frac{2 - \alpha}{2}}
    |g_0| \big( C_0^{\frac{\alpha}{2}} 
    + T^{\frac{\alpha}{2}} \|X
    \|_{C^{\frac{1}{2}}([0,T];L^2(\Omega;\R^{d}))}^\alpha \big)
  \end{align*}
  with $C_0 = 2 \E [ |X_0|^2 ] + 4 T ( \Phi(0) + |g_0|^2 )$.
 
  Observe that, since we avoid the use of Gronwall-type
  inequalities, the error constant does not grow exponentially with
  time $T$. This indicates that the backward Euler--Maruyama method is
  particularly suited for long-time simulations as is often required
  in Markov-chain Monte Carlo methods, for example, in the unadjusted
  Langevin algorithm \cite{roberts1996}.
\end{remark}

%
%

\section{Properties of the exact solution}
\label{sec:exact}

In this section, we turn our attention to the multi-valued stochastic 
differential equation (MSDE)  in \eqref{eq1:SDE}. We give a complete 
account of the assumptions
imposed on the coefficient functions. In addition, we collect some results on
the existence and uniqueness of a strong solution to the MSDE. We also include
useful results on higher moment bounds of the exact solution. 


\begin{assumption}
  \label{ass:f}
  The set valued mapping $f \colon \R^d \to 2^{\R^d}$ is maximal
  monotone with $\operatorname{int} D(f) \neq \emptyset$.  Moreover,
  there exist constants $\beta, \lambda \in [0,\infty)$,
  $\mu \in (0,\infty)$, and $p \in [1,\infty)$ such that
  \begin{align*}
    \inner{f_v}{v} \geq \mu |v|^p - \lambda
    \quad\text{and}\quad
    |f_v| \leq \beta(1 + |v|^{p-1})
  \end{align*}
  for every $v \in D(f)$ and $f_v
  \in f(v)$.
\end{assumption}

\begin{assumption}
  \label{ass:b}
  The function $b \colon \R^d \to \R^d$ is Lipschitz continuous; i.e., 
  there exists a constant $L_b \in [0,\infty)$ such that
  \begin{align*}
    |b(v) - b(w)| \leq L_b |v-w|
  \end{align*} 
  for all $v,w \in \R^d$.
\end{assumption}

\begin{assumption}
  \label{ass:g}
  The function $g \colon \R^d \to \R^{d,m}$ is Lipschitz continuous; i.e., 
  there exists a constant $L_g \in [0,\infty)$ such that
  \begin{align*}
    |g(v) - g(w)| \leq L_g |v-w|
  \end{align*} 
  for all $v,w \in \R^d$.
\end{assumption}

\begin{assumption}
  \label{ass:X0}
  The initial value $X_0$ is an $\F_0$-measurable and
  $D(f)$-valued random variable. Furthermore,
  \begin{align*}
    \E[ |X_0|^{\max(2p - 2,2)} ] < \infty,
  \end{align*}
  where the value of $p$ is the same as 
  in Assumption~\ref{ass:f}.
\end{assumption}

Observe that Assumptions~\ref{ass:b} and \ref{ass:g} directly imply
that $b$ and $g$ grow at most linearly. More precisely, after possibly
increasing the values of $L_b$ and $L_g$, we obtain the bounds
\begin{align}
  \label{eq4:lingrowthbg}
  |b(v)| \le L_b (1 + |v|),\quad
  |g(v)| \le L_g (1 + |v|),
\end{align}
for all $v \in \R^d$. 

\begin{remark}  \label{rem:zeroindf}
  Without loss of generality we will assume that $0 \in D(f)$. 
  Otherwise, since the graph of $f$ is not empty, we take $v_0 \in D(f)$ and
  $f_{v_0} \in f(v_0)$ and replace $f$, $b$, and $g$ by suitably
  shifted mappings, for instance, $\tilde{f}(v) := f(v + v_0)$. Then
  $0 \in D(\tilde{f})$ holds. Compare further with \cite[Abschn.~3.3.3]{ruzicka2004}.
\end{remark}

Next, we introduce the notion of a solution of \eqref{eq1:SDE}, 
which we use for the remainder of this paper.

\begin{definition}
  \label{def:solution}
  A tuple $(X,\eta)$ is called a solution of the multi-valued stochastic
  differential equation \eqref{eq1:SDE}, if the following conditions
  hold. 
  \begin{itemize}
    \item[(i)] The mapping $X \colon [0,T] \times \Omega \to \R^d$ is
      an $(\F_t)_{t \in[0,T]}$-adapted, almost surely continuous
      stochastic process such that $X(t) \in
      \overline{D(f)}$ for all $t \in (0,T]$ with probability one.
    \item[(ii)] The mapping $\eta \colon [0,T] \times \Omega \to \R^d$ 
      is an $(\F_t)_{t \in[0,T]}$-adapted stochastic process such that
      \begin{align*}
        \int_0^T |\eta(t)| \diff{t} < \infty, \quad \P\text{-almost surely.}
      \end{align*}
    \item[(iii)] The equality 
      \begin{align}        
        X(t) + \int_{0}^{t} \eta(s) \diff{s} = X_0 + \int_{0}^{t} b(X(s))
        \diff{s} + \int_{0}^{t} g(X(s)) \diff{W(s)} 
        \label{eq:sol}
      \end{align}
      holds for all $t \in [0,T]$ and $\P$-almost surely.
    \item[(iv)] For almost all $\omega \in \Omega$ and $t \in [0,T]$,
      it follows that $\eta(t,\omega) \in f(X(t,\omega))$; in other words, for 
      every $y \in D(f)$ and $f_y \in f(y)$ the inequality
    \begin{align*}
      \inner{\eta(t) - f_{y}}{X(t) - y} \geq 0
    \end{align*}
  is satisfied  for almost every $t \in [0,T]$ and $\P$-almost
  surely, cf.~Definition~\ref{def:mono}.
  \end{itemize}
\end{definition}

This notion of a solution has been considered in, for example, 
\cite{Cepa1995}, \cite{kree1982}, \cite{pettersson1995}, and 
\cite{stephanPhd}, where also the existence of a unique solution 
is shown. Due to their importance for the error analysis, we next
prove certain moment estimates.

\begin{theorem}
  \label{thm:exact}
  Let Assumptions~\ref{ass:f} and \ref{ass:X0} be satisfied with $p \in
  [1,\infty)$.
  Then there exists a unique solution $(X,\eta)$ of \eqref{eq1:SDE} 
  in the sense of Definition~\ref{def:solution}.   There is a
  constant $C\in(0,\infty)$ such that 
  \begin{align*}
    \sup_{t \in [0,T]} \E \big[ |X(t)|^2 \big] +
    \E \Big[ \int_0^T |X(s)|^p \diff{s} \Big] \le C. 
  \end{align*}
  Furthermore, if $p \in (1,\infty)$ and $\frac{1}{p}+\frac{1}{q} = 1$, then 
    \begin{align*}
    \E \Big[ \int_0^T |\eta(s)|^q \diff{s} \Big] \le C.  
  \end{align*}
\end{theorem}

\begin{proof}
  Existence and uniqueness is shown, for instance, in \cite{kree1982}.
  For 
  \begin{align*}
    X(t) = X_0 + \int_{0}^{t} (b(X(s)) - \eta(s)) \diff{s}
    + \int_{0}^{t} g(X(s)) \diff{W(s)}
  \end{align*}
  the equality
  \begin{align*}
    |X(t)|^2 
    &= |X_0|^2 + \int_{0}^{t} \big(2 \inner{b(X(s))}{X(s)}
    - 2 \inner{\eta(s)}{X(s)} 
    + |g(X(s))|^2\big) \diff{s}\\
    &\quad + \int_{0}^{t} 2\inner{X(s)}{g(X(s))\diff{W(s)}},
  \end{align*}
  holds by an application of It\=o's formula (see \cite[Chap.~4.7, 
  Theorem~7.1]{friedman1975}).
  From the coercivity assumption on $f$ we obtain that
  \begin{align*}
    \inner{f_{X(s)}}{X(s)} \geq \mu |X(s)|^p - \lambda
  \end{align*}
  for every $f_{X(s)} \in f(X(s))$ and almost every $s \in [0,T]$. The fact  
  that $\eta(s) \in f(X(s))$ for almost every $s \in [0,T]$ then implies that
  \begin{align*}
    \int_{0}^{t} \inner{\eta(s)}{X(s)} \diff{s}
    \geq  \mu \int_{0}^{t} |X(s)|^p \diff{s} - \lambda t.
  \end{align*}
  Since $b$ and $g$ satisfies the linear growth bound 
  \eqref{eq4:lingrowthbg}, we have
    \begin{align*}
    \int_{0}^{t} \inner{b(X(s))}{X(s)}\diff{s}
    \leq 2 L_b \int_{0}^{t} \big(1 +|X(s)|^2\big) \diff{s}
  \end{align*}
  as well as
  \begin{align*}
    \int_{0}^{t} |g(X(s))|^2 \diff{s}
    \le 2 L_g^2 \int_{0}^{t} \big( 1+  |X(s)|^2\big) \diff{s}.
  \end{align*}
  Thus, we get
  \begin{align*}
    &|X(t)|^2 + 2 \mu \int_{0}^{t} |X(s)|^p \diff{s} \\
    &\quad \leq |X_0|^2 + \big(4 L_b + 2 L_g^2 \big) 
    \int_{0}^{t} \big(1 +  |X(s)|^2 \big) \diff{s}
    + 2 \lambda t 
    + \int_{0}^{t} 2\inner{X(s)}{g(X(s)) \diff{W(s)}}.
  \end{align*}
  We introduce
  \begin{align*}
    Z(t) &:= |X(t)|^2 + 2 \mu \int_{0}^{t} |X(s)|^p \diff{s}, \quad
    M(t) := \int_{0}^{t} 2\inner{X(s)}{g(X(s)) \diff{W(s)}}, \\
    \xi(t) &:= |X_0|^2 + 2 ( \lambda + 2 L_b + L_g^2 ) t,
    \quad \varphi(t): = 4 L_b + 2 L_g^2.
  \end{align*}
  Then $Z$, $M$, and $\xi$ are $(\F_t)_{t\in[0,T]}$-adapted and almost surely 
  continuous stochastic processes. Furthermore,  $M$ is a local 
  $(\F_t)_{t\in[0,T]}$-martingale with $M(0) = 0$. Thus, an application of 
  Lemma~\ref{lem:StochGronwall} yields, for every $t \in [0,T]$, that
  \begin{align*}
    \E \big[Z(t) \big] 
    &\leq \exp\Big(\int_{0}^{t} \varphi(s) \diff{s} \Big) 
    \E\big[ \sup_{s\in[0,t]} \xi(s) \big]\\
    &= \exp\big( (4 L_b + 2 L_g^2) t \big) \big( 
    \E\big[|X_0|^2\big]
    + 2 ( \lambda + 2 L_b + L_g^2 ) t \big).
  \end{align*}
  Inserting the definition of $Z$ then proves the first 
  estimate.
  
  Furthermore, if Assumption~\ref{ass:f} holds with $p \in (1,\infty)$, then
  we have, for every $f_{x} \in f(x)$, $x \in \R^d$, that
  \begin{align*}
    |f_{x} | \leq \beta(1 + |x|^{p-1}),
  \end{align*}
  with $q = \frac{p}{p-1}$. Therefore, it follows that 
  \begin{align*}
    \Big(\int_{0}^{T} \E\big[ |\eta(s)|^q \big] \diff{s}\Big)^{\frac{1}{q}}
    \leq T^{\frac{1}{q}} \beta
    + \beta \Big(\int_{0}^{T} \E\big[ |X(s)|^p \big] \diff{s}\Big)^{\frac{1}{q}}
    \leq C
  \end{align*}
  since $\eta(s) \in f(X(s))$ for almost every $s \in [0,T]$.
\end{proof}

\begin{remark}  \label{rem:differentsolution}
  Let us mention that, for instance, in \cite[Chapter~4]{pardoux2014} and the
  references therein, a weaker notion of a solution to \eqref{eq1:SDE} is
  found. More precisely, if $(X,\eta)$ is a solution in the
  sense of Definition~\ref{def:solution}, then $(X,H)$ is a solution in the sense
  of \cite[Chapter~4]{pardoux2014} with the definition
  \begin{align*}
    H(t) := \int_0^t \eta(s) \diff{s}, \quad t \in [0,T].
  \end{align*}  
  In particular, the process $H$ is a continuous, progressively
  measurable process with bounded total variation and $H(0) = 0$ almost surely.
  The stronger condition of absolute continuity of the process $H$,
  which is required in Definition~\ref{def:solution},  
  is essential in the proof of Theorem~\ref{thm:conv} below.
  This explains why we work with the stronger notion of a solution in
  Definition~\ref{def:solution}.
\end{remark}

%
%

\section{Well-posedness of the backward Euler method}
\label{sec:wellposed}

In this section, we show that the backward Euler--Maruyama method  
\eqref{eq1:multiEuler} for the MSDE \eqref{eq1:SDE} is well-posed
under the same assumptions as in the previous section.

\begin{lemma}
  \label{lem:solvableF}
  Let Assumptions~\ref{ass:f} and \ref{ass:b} be satisfied.
  Furthermore, let $w \in \R^d$ and $k \in (0,T]$ be given with
  $L_b k \in [0,1)$. Then there exist uniquely determined
  $x_0 \in D(f)$ and $\eta_{x_0} \in f(x_0)$, which satisfy the
  nonlinear equation
  \begin{align} 
    \label{eq2:nonlinearF}
    x_0 + k \eta_{x_0} - k b(x_0) = w.
  \end{align}
\end{lemma}

\begin{proof}
  We first show that there exists a unique $x_0\in D(f)$ such that
  \begin{align}\label{eq2:inclusionF}
    x_0 + k f(x_0) - k b(x_0) = (\id + k f - k b)(x_0)\ni w.
  \end{align}
  To this end, notice that for all $x,y \in \R^d$, the inequalities
  \begin{align*}
    \inner{(\id -k b)x - (\id -k b)y}{x-y}
    \geq |x-y|^2 - k L_b |x-y|^2 \geq 0
  \end{align*}
  hold due to the step-size bound. In addition, it follows from
  \eqref{eq4:lingrowthbg} that
  \begin{align*}
    \frac{\inner{(\id -k b)x}{x}}{|x|} = \frac{|x|^2 - k \inner{b(x)}{x}}{|x|}
    \ge (1 - k L_b) |x| - k L_b 
  \end{align*}
  for all $x \in \R^d$. Hence, $(\id + k f - k b)$ is the sum 
  of the maximal monotone operator $kf$ and the mapping $(\id - k b)$, which is
  single-valued, Lipschitz continuous, monotone and coercive. 
  
  Thus, we can apply \cite[Theorem~2.1]{barbu.2010} and obtain the existence of
  $x_0 \in D(f)$ such that \eqref{eq2:inclusionF} holds. Furthermore, there 
  necessarily exists a corresponding unique element $\eta_{x_0} \in f(x_0)$ with
  \begin{align*}
    \eta_{x_0} = \frac{1}{k} (w - x_0) + b(x_0).
  \end{align*}
  It remains to prove the uniqueness of $x_0$, which directly implies the 
  uniqueness of $\eta_{x_0}$. Assume that there exist $x_1 \in D(f)$ and 
  $\eta_{x_1} \in f(x_1)$ as well as $x_2 \in D(f)$ and $\eta_{x_2} \in f(x_2)$
  such that 
  \begin{align*}
    x_1 + k \eta_{x_1} - k b(x_1) = w, \quad 
    x_2 + k \eta_{x_2} - k b(x_2) = w.
  \end{align*}
  By considering the difference of these equations tested with $x_1- 
  x_2$, we obtain
  \begin{align*}
    0 &= \inner{x_1 - x_2}{x_1 - x_2} + k \inner{\eta_{x_1} - \eta_{x_2}}{x_1 - 
    x_2} - k \inner{b(x_1) - b(x_2) }{x_1 - x_2} \\
    &\geq |x_1 - x_2|^2 - k L_b |x_1 - x_2|^2 \geq 0.
  \end{align*}
  Since $1 - k L_b > 0$ we must have $x_1 = x_2$ and the proof is complete.
\end{proof}

For later use, we note that the solution operator for
\eqref{eq2:nonlinearF} is Lipschitz continuous.

\begin{lemma}
  \label{lem:stab}
  Let Assumptions~\ref{ass:f} and \ref{ass:b} be satisfied. 
  For $k \in (0,T]$ with $L_b k \in [0,1)$ let $S_k \colon \R^d \to 
  D(f)$ be the solution operator that maps $w \in \R^d$ to the unique 
  solution $x_0 \in  D(f)$ of \eqref{eq2:nonlinearF}. Then $S_k$ is 
  globally Lipschitz continuous with
  \begin{align*}
    | S_k(w_1) - S_k(w_2)| \le \frac{1}{1 - k L_b} |w_1 - w_2|
    \quad \text{for all } w_1, w_2 \in \R^d. 
  \end{align*}
\end{lemma}

\begin{proof}
  Let $w_1, w_2 \in \R^d$ and $k \in (0,T]$ with $L_b k \in [0,1)$  be 
  given. 
  Let $x_i = S_k(w_i) \in D(f)$ and $\eta_{x_i} \in f(x_i)$, $i\in \{1,2\}$,
  denote the unique solutions of the equations
  \begin{align*}
    x_1 + k \eta_{x_1} -k b(x_1) = w_1,\quad 
    x_2 + k \eta_{x_2} -k b(x_2) = w_2.
  \end{align*}
  By considering the difference of these equations, tested with $x_1 - 
  x_2$, we obtain 
  \begin{align*}
    &|x_1 - x_2|^2 + k \inner{\eta_{x_1} - \eta_{x_2}}{x_1 - x_2} 
    - k \inner{b(x_1) - b(x_2 )}{x_1 - x_2} \\
    &\quad = \inner{w_1 - w_2}{x_1 - x_2}.
  \end{align*}
  By using the Cauchy--Schwarz inequality for the right-hand side as well as 
  the monotonicity and the Lipschitz continuity for the left-hand
  side, we get
  \begin{align*}
    (1 - k L_b) |x_1 - x_2|^2 
    \leq |w_1 - w_2| |x_1 - x_2|.
  \end{align*}
  Reinserting $x_i = S_k(w_i)$ then shows that
  \begin{align*}
     |S_k(w_1) - S_k(w_2)| = |x_1 - x_2|
    \leq \frac{1}{1 - k L_b} |w_1 - w_2|
  \end{align*}
  as claimed.
\end{proof}

\begin{theorem}
  \label{thm:existence_multi} 
  Let Assumptions~\ref{ass:f} to \ref{ass:X0} be satisfied.
  Then for every step size $k = \frac{T}{N}$, $N \in \N$, with $L_b k \in 
  [0,1)$ there exist uniquely determined families of  
  square-integrable, $\R^d$-valued and $(\F_{t_n})_{n \in 
  \{0,\ldots,N\}}$-adapted random variables $(X^n)_{n\in \{0,\dots,N\}}$ and 
  $(\eta^n)_{n\in \{0,\dots,N\}}$ such that  $X^n \in D(f)$, $\eta^n \in f(X^n)$ 
  for every $n \in \{0,\dots,N\}$ and
  \begin{align}
    \label{eq6:multiEuler}
    X^n + k \eta^n = X^{n-1} + k b(X^n) + g(X^{n-1})\Delta W^n 
  \end{align}
for every $n \in \{1,\dots,N\}$,  $\P$-almost surely. 
\end{theorem}

\begin{proof}
  We prove the existence of $(X^n)_{n\in \{0,\dots,N\}}$ and $(\eta^n)_{n\in 
  \{0,\dots,N\}}$ by induction over $n \in \{0,\ldots,N\}$. From the 
  assumptions on $X_0$ and $f$ it is clear that $X^0=X_0$ and $\eta^0 \in
  f(X_0)$ are $\F_{t_0}$-adapted and square-integrable. In particular,
  it follows from Assumptions~\ref{ass:f} and \ref{ass:X0} that
  \begin{align*}
    \E \big[ |\eta^0 |^2 \big] \le  \beta^2 \E \big[ (1 + |X_0|^{p-1})^2
    \big] \le 2 \beta^2 \big( 1 + \E [ |X_0|^{2p-2} ] \big).
  \end{align*}
  
  Next, we assume that $(X^j)_{j \in \{0,\ldots,n-1\}}$ and $(\eta^j)_{j \in 
  \{0,\ldots,n-1\}}$ are $(\F_{t_{j}})_{j \in \{0,\ldots,n-1\}}$-adapted, 
  square-integrable and satisfy \eqref{eq6:multiEuler} for all $j \in
  \{1,\ldots,n-1\}$. By Lemma~\ref{lem:solvableF} there exist uniquely
  determined $X^n(\omega) \in D(f)$ and $\eta^n(\omega) \in f(X^n(\omega))$ for
  almost every $\omega \in \Omega$ such that
  \begin{align*}
    X^n(\omega) + k \eta^n(\omega) = X^{n-1}(\omega) + k b(X^n(\omega)) + 
    g(X^{n-1}(\omega))\Delta W^n(\omega).
  \end{align*}
  By Lemma~\ref{lem:stab}, the solution 
  operator $S_{k} \colon \R^d \to D(f)$ that maps $X^{n-1}(\omega) + 
  g(X^{n-1}(\omega))\Delta W^n(\omega)$ to $X^n(\omega) \in D(f)$ is Lipschitz continuous. 
  As $S_{k}$ is Lipschitz continuous and, hence, of linear growth it
  follows that  $X^n$ is an $\F_{t_n}$-measurable and
  square-integrable random 
  variable. To be more
  precise, we have the bound
  \begin{align*}
    \big\| X^n \big\|_{L^2(\Omega;\R^d)} 
    &=\big\| S_{k}(X^{n-1} + g(X^{n-1}) \Delta W^n) \|_{L^2(\Omega;\R^d)}\\
    &\le | S_{k}(0) | + \big\| X^{n-1} + g(X^{n-1}) \Delta W^n
    \big\|_{L^2(\Omega;\R^d)}. 
  \end{align*}  
  This implies, in particular, that
  \begin{align*}
     \eta^n = - \frac{1}{k}\big(X^n - X^{n-1}\big) + b(X^n) 
     + g(X^{n-1})\frac{\Delta W^n}{k} \quad \text{a.s. in } \Omega
  \end{align*}
  is also a $\F_{t_n}$-measurable and square-integrable random variable as
  $X^n$, $X^{n-1}$ and $g(X^{n-1})\Delta W^n$ have these properties. This
  finishes the proof of the induction and hence that of the theorem. 
\end{proof}

Next we state an \emph{a priori} estimate for the sequence of random 
variables satisfying recursion \eqref{eq1:multiEuler}. 

\begin{lemma}
  \label{lem:apriori}
  Let Assumptions~\ref{ass:f} to \ref{ass:X0} be satisfied.
  For a step size $k = \frac{T}{N}$, $N \in \N$, 
  with $5 L_b k \in [0,1)$, let $(X^n)_{n \in \{0,\dots,N\}}$ and
  $(\eta^n)_{n\in \{0,\dots,N\}}$ be two families of $(\F_{t_n})_{n \in 
  \{0,\ldots,N\}}$-adapted random variables as stated in 
  Theorem~\ref{thm:existence_multi}.
  Then there exists $K_X \in (0,\infty)$ independent of the step size
  $k = \frac{T}{N}$ such that
  \begin{align}
    \label{eq4:aprioriX}
    \begin{split}
      &\max_{n \in \{1,\ldots,N\}} \E\big[ |X^n|^2 \big] 
      + \frac{1}{2} \sum_{ j = 1}^N \E \big[ | X^{j} - X^{j-1} |^2 \big]
      + 2 \mu k \sum_{ j = 1}^N \E \big[|X^j|^p\big] \leq K_X.
    \end{split}
  \end{align}
  In addition, if $p \in (1,\infty)$, then there exists $K_\eta \in (0,\infty)$ 
  independent of the step size $k = \frac{T}{N}$ such that
  \begin{align}
    \label{eq4:aprioriEta}
    k \sum_{ j = 1}^N \E \big[|\eta^j|^q\big] 
    \leq K_{\eta}, 
  \end{align}
  where $q \in (1,\infty)$ is given by $\frac{1}{p} + \frac{1}{q} = 1$. 
\end{lemma}

\begin{remark} \label{rem:p1}
  If $p = 1$ in Assumption~\ref{ass:f}, then $f$ and, hence, $(\eta^n)_{n \in
  \{1,\ldots,N\}}$ are bounded. In particular, \eqref{eq4:aprioriEta} 
  holds for any $q \in (1,\infty)$ and for any step size $k = \frac{T}{N}$ with
  $L_b k \in [0,1)$. 
\end{remark}

\begin{proof}[Proof of Lemma~\ref{lem:apriori}]
  First, we recall the identity
  \begin{align*}
    \inner{X^n - X^{n-1}}{X^n}
    = \frac{1}{2} \big( |X^n|^2 - |X^{n-1}|^2 + |X^n - X^{n-1}|^2 \big).
  \end{align*}
  As $\eta^n \in f(X^n)$, using Assumptions~\ref{ass:f} and 
  \ref{ass:b}, it follows that
  \begin{align*}
    &\frac{1}{2} \big( |X^n|^2 - |X^{n-1}|^2 + |X^n - X^{n-1}|^2 \big)
    + k \mu |X^n|^p
    \\
    &\quad \leq \langle  X^n - X^{n-1}, X^n \rangle +  k \inner{\eta^n}{X^n} 
     + k \lambda\\
     &\quad = k \inner{b(X^n)}{X^n} + \langle g(X^{n-1}) \Delta W^n, X^n
     \rangle + k \lambda\\
     &\quad \le k L_b (1 + |X^n|)|X^n|  + \langle g(X^{n-1}) \Delta W^n, X^n
     \rangle + k \lambda,
  \end{align*}
  where we also applied \eqref{eq4:lingrowthbg}. 
  Hence,  
  \begin{align*}
    &\frac{1}{2} \big( |X^n|^2 - |X^{n-1}|^2 + |X^n - X^{n-1}|^2 \big) + k \mu
    |X^n|^p \\
    &\quad \leq k (\lambda + L_b) + \frac{5}{4} k L_b |X^n|^2 + \langle
    g(X^{n-1}) \Delta W^n, X^n - X^{n-1} \rangle\\
    &\qquad + \langle g(X^{n-1}) \Delta W^n, X^{n-1} \rangle \\
    &\quad \leq  k (\lambda + L_b) + \frac{5}{4} k L_b |X^n|^2 +
    \big| g(X^{n-1}) \Delta W^n \big|^2
    + \frac{1}{4} | X^n - X^{n-1}|^2\\
    &\qquad +  \langle g(X^{n-1}) \Delta W^n, X^{n-1} \rangle,
  \end{align*}
  for every $n \in \{1,\ldots,N\}$, where we also applied the Cauchy--Schwarz
  and weighted Young inequalities.
  After a kick-back,
  we sum from $1$ to  $n \in \{1,\ldots,N\}$ to obtain
  \begin{align*}
    &|X^n|^2 + \frac{1}{2} \sum_{ j = 1}^n | X^{j} - X^{j-1} |^2
    + 2 k \mu \sum_{j = 1}^n |X^j|^p \\
    &\quad \leq |X^0|^2 + 2 (\lambda + L_b) T 
    + \frac{5}{2} k L_b \sum_{j = 1}^n |X^j|^2
    + 2 \sum_{j = 1}^n \big| g(X^{j-1}) \Delta W^j \big|^2\\
    &\qquad + 2 \sum_{j = 1}^n \big\langle g(X^{j-1}) \Delta W^j, X^{j-1} 
    \big\rangle.
  \end{align*}
  After taking expectations, the last term on the right-hand side vanishes. 
  Then, applications of It\=o's isometry and \eqref{eq4:lingrowthbg} give
  \begin{align*}
    &\E\big[ |X^n|^2 \big] + \frac{1}{2} \sum_{ j = 1}^n \E \big[ | X^{j} -
    X^{j-1} |^2 \big]
    + 2 k \mu \sum_{j = 1}^n \E \big[|X^j|^p\big] \\
    &\quad \le \E \big[ |X^0|^2 \big] + 2 (\lambda + L_b) T
    + \frac{5}{2} k L_b \sum_{j = 1}^n \E\big[ |X^j|^2\big]
    + 2 \sum_{j = 1}^n \E \big[ \big| g(X^{j-1}) \Delta W^j \big|^2 \big] \\
    &\quad \le \E \big[ |X^0|^2 \big] + 2 (\lambda + L_b) T 
    + \frac{5}{2} k L_b \sum_{j = 1}^n \E\big[ |X^j|^2\big]
    + 2 k \sum_{j = 1}^n \E\big[ | g(X^{j-1})|^2 \big]\\
    &\quad \le (1 + 4 k L_g^2 ) \E \big[ |X^0|^2 \big] 
    + 2 (\lambda + L_b + 2 L_g^2) T + k \Big( \frac{5}{2} L_b + 4 L_g^2\Big)
    \sum_{j = 1}^{n-1} \E\big[ |X^j|^2\big]\\
    &\qquad+ \frac{5}{2} k L_b \E\big[ |X^n|^2\big].
  \end{align*}
  Since the step-size bound $5 L_b k \in [0,1)$ ensures that 
  \begin{align*}
    1 - \frac{5}{2} k L_b > \frac{1}{2},
  \end{align*}
  the discrete Gronwall inequality (see, for example, \cite{clark1987}) is 
  applicable and completes 
  the proof of \eqref{eq4:aprioriX}. 
  Finally, it follows from the polynomial growth bound on
  $f$ that
  \begin{align*}
    \Big(k \sum_{j = 1}^N \E \big[|\eta^j|^q\big]\Big)^{\frac{1}{q}}
    \leq \Big(k \sum_{j = 1}^N \E \big[\beta^q (1 + 
    |\eta^j|^{p-1})^q\big]\Big)^{\frac{1}{q}}
    \leq \beta T^{\frac{1}{q}}
    + \beta \Big(k \sum_{j = 1}^N \E \big[ |X^j|^p \big] \Big)^{\frac{1}{q}},
  \end{align*}
  and an application of \eqref{eq4:aprioriX} then yields \eqref{eq4:aprioriEta}.
\end{proof}

%
%

\section{Error estimates in the general case}
\label{sec:aposteriori}

In this section, we derive an error estimate for the backward Euler 
method given by \eqref{eq1:multiEuler} for the MSDE
\eqref{eq1:SDE}. 

To prove the convergence of the scheme \eqref{eq1:multiEuler}, let
us fix some notation. Throughout this section, we assume that the equidistant 
step size $k = \frac{T}{N}$ is small enough so that the \emph{a priori}
estimates in Lemma~\ref{lem:apriori} hold. Further, as in \eqref{eq2:Xlin} and
\eqref{eq2:Xconst}, we denote the piecewise linear interpolants of the discrete
values by $\X(0) = X^0$, $\Eta(0) = \eta^0$ for $\eta^0 \in f(X^0)$ and   
\begin{align*}
  \X(t) := \frac{t-t_{n-1}}{k} X^n + \frac{t_n-t}{k} X^{n-1},
  & &\Eta(t) := \frac{t-t_{n-1}}{k} \eta^n + \frac{t_n-t}{k} \eta^{n-1}
\end{align*}
for all $t \in (t_{n-1},t_n]$ and $n \in \{1,\ldots,N\}$.
Similarly, we define the piecewise constant interpolant by
$\overline{\X}(0)=\underline{\X}(0)=X^0$ and
\begin{align*}
  &\overline{\X}(t) = X^n
  \quad \text{and} \quad
  \underline{\X}(t) = X^{n-1}, \text{ as well as}
  &&\overline{\Eta}(t) = \eta^n
  \quad \text{and} \quad
  \underline{\Eta}(t) = \eta^{n-1},
\end{align*}
for all $t \in (t_{n-1},t_n]$ and $n \in \{1,\ldots,N\}$.
Moreover, we introduce the stochastic processes $G \colon [0,T] \times
\Omega \to \R^d$ and $\G \colon [0,T] \times \Omega \to \R^d$ defined by
\begin{align}
  \label{eq4:G}
  G(t) = \int_{0}^{t} g(X(s)) \diff{W(s)}, \quad \text{ for all } t \in [0,T],
\end{align}
as well as by $\G(0) = 0$ and, for all $n \in \{1,\ldots,N\}$ and $t \in
(t_{n-1},t_n]$, 
\begin{align}
  \label{eq4:defG}
  \begin{split}
    \G(t) &= \frac{t - t_{n-1}}{k}g(X^{n-1}) \Delta W^n
    + \sum_{i=1}^{n-1} g(X^{i-1}) \Delta W^i\\
    &= \frac{t - t_{n-1}}{k} g(X^{n-1}) \Delta W^n
    + \int_{0}^{t_{n-1}} g(\underline{\X}(s)) \diff{W(s)}.
  \end{split}
\end{align}
In view of \eqref{eq1:multiEuler} and the definition of $\G$ for $t \in 
(t_{n-1},t_n]$, $n \in \{1,\dots,N\}$, we obtain the representation 
\begin{align}
  \label{eq6:Xlin}
  \begin{split}
    \X(t)
    &= X^{n-1} + \frac{t - t_{n-1}}{k} \big(X^n - X^{n-1}\big)\\
    &= \Big( X^0 + k \sum_{i=1}^{n-1} \big(b(X^i) - \eta^i\big) 
    + \sum_{i=1}^{n-1} g(X^{i-1}) \Delta W^i \Big) \\
    &\quad + \frac{t - t_{n-1}}{k} \Big(k b(X^n) - k \eta^n + g(X^{n-1}) \Delta
    W^n\Big)\\
    &=  X_0 + \int_{0}^{t} \big(b(\overline{\X}(s)) - \overline{\Eta}(s)\big) 
    \diff{s} 
    + \G(t).
  \end{split}
\end{align}

We begin the derivation of our error estimate by
considering the difference between the stochastic integral $G$ and its
approximation $\G$.

\begin{lemma}
  \label{lem:Gerror}
  Let Assumptions~\ref{ass:f} to \ref{ass:X0} be satisfied.
  Then there exists $K_G \in (0,\infty)$ such that, for every
  equidistant step size $k = \frac{T}{N}$, $N \in \N$ with 
  $5 L_b k \in [0,1)$ and every $t \in [0,T]$, we have 
  \begin{align}
    \label{eq6:term3}
    \begin{split}
      \big\| G(t) - \G(t) \big\|_{L^2(\Omega;\R^d)}^2 \le
      K_G k + 2L_g^2 \int_0^t \E \big[ | X(s) - \X(s) |^2 \big] \diff{s}.
    \end{split}
  \end{align}
  In addition, for every $\rho \in [2,\infty)$, there exists $K_\rho \in
  (0,\infty)$ such that, for every $n \in \{1,\ldots,N\}$ and $t \in
  (t_{n-1},t_n]$, the following estimates hold:
  \begin{align}
    \label{eq6:Greg1}
    \Big( \int_{t_{n-1}}^{t} \E \big[ | G(t) - G(s) |^\rho \big]
    \diff{s} \Big)^{\frac{1}{\rho}}
    &\le K_\rho k^{\frac{1}{2}} \Big( \int_{t_{n-1}}^{t} 
    \big(1 + \E \big[ | X(s) |^\rho \big] \big)
    \diff{s} \Big)^{\frac{1}{\rho}}
  \end{align}
  and
  \begin{align}
    \label{eq6:Greg2}
    \sup_{s \in [t_{n-1},t]} \| \G(t) - \G(s) \|_{L^\rho(\Omega;\R^d)}^{\rho} 
    &\le K_\rho k^{\frac{\rho}{2}} \big( 1 +
    \|X^{n-1}\|_{L^\rho(\Omega;\R^d)}^{\rho}\big).
    \end{align}
\end{lemma}

\begin{proof}
  Recall the definitions of $G$ and $\G$ from \eqref{eq4:G} and
  \eqref{eq4:defG}. First, we add and subtract a term and then apply the triangle
  inequality. Then, for every $n \in \{1,\ldots,N\}$ and $t \in
  (t_{n-1},t_n]$ we arrive at 
  \begin{align*}
    &\big\| G(t) - \G(t) \big\|_{L^2(\Omega;\R^d)}
    \le \Big\| \int_0^t \big( g(X(s)) - g(\underline{\X}(s)) \big)
    \diff{W(s)} \Big\|_{L^2(\Omega;\R^d)}\\
    &\qquad + \Big\| \int_{t_{n-1}}^t  g(\underline{\X}(s)) \diff{W(s)}
    - \frac{t - t_{n-1} }{k} g(X^{n-1}) \Delta W^n \Big\|_{L^2(\Omega;\R^d)}
    \\ 
    &\quad = \Big( \int_0^t \E \big[ | g(X(s)) - g(\underline{\X}(s)) |^2 \big]
    \diff{s} \Big)^{\frac{1}{2}}\\
    &\qquad + \Big\| g(X^{n-1}) \Big( \frac{t_n - t}{k} \big( W(t) - W(t_{n-1})
    \big) - \frac{t - t_{n-1} }{k} \big(W(t_n) - W(t) \big)\Big)
    \Big\|_{L^2(\Omega;\R^d)}
  \end{align*}
  by an application of It\=o's isometry. Furthermore, due to the Lipschitz
  continuity of $g$ we obtain
  \begin{align*}
    & \Big( \int_0^t \E \big[ | g(X(s)) - g(\underline{\X}(s)) |^2 \big] \diff{s}
    \Big)^{\frac{1}{2}}\\
    &\quad \le L_g \Big( \int_0^t \E \big[ | X(s) - \X(s) |^2 \big] \diff{s}
    \Big)^{\frac{1}{2}}
    + L_g \Big( \int_0^t \E \big[ | \X(s) - \underline{\X}(s) |^2 \big]
    \diff{s} \Big)^{\frac{1}{2}}\\ 
    &\quad \le L_g \Big( \int_0^t \E \big[ |X(s) - \X(s)|^2 \big]
    \diff{s} \Big)^{\frac{1}{2}}
    + L_g \Big( \frac{1}{3} k \sum_{i = 1}^n \E\big[ |X^i - X^{i-1}|^2 \big]
    \Big)^{\frac{1}{2}},
  \end{align*}
  where the last step follows from the identity
  \begin{align*}
    \X(s) - \underline{\X}(s)
    = \frac{s-t_{i-1}}{k} X^i + \frac{t_i-s}{k} X^{i-1} - X^{i-1}
    = \frac{s-t_{i-1}}{k} \big(X^i - X^{i-1}\big)
  \end{align*}
  which holds for every $s \in (t_{i-1},t_i]$, $i \in \{1,\ldots,N\}$.
  Finally, it follows from the same arguments as in the proof of
  Lemma~\ref{lem:Werror} and by \eqref{eq4:lingrowthbg} for every 
  $t \in (t_{n-1},t_n]$ that
  \begin{align*}
    &\Big\| g(X^{n-1}) \Big( \frac{t_n - t}{k} \big( W(t) - W(t_{n-1})
    \big) - \frac{t - t_{n-1} }{k} \big(W(t_n) - W(t) \big)\Big)
    \Big\|_{L^2(\Omega;\R^d)}^2\\
    &\quad = \frac{1}{k^2} \big(
    (t_n - t)^2 (t - t_{n-1}) + (t - t_{n-1})^2 (t_n - t)\big)
    \big\| g(X^{n-1}) \big\|_{L^2(\Omega;\R^d)}^2\\
    &\quad = \frac{1}{k} (t_n - t) (t - t_{n-1}) 
    \big\| g(X^{n-1}) \big\|_{L^2(\Omega;\R^d)}^2\\
    &\quad \le \frac{1}{4} L_g^2 k 
    \big(1+ \| X^{n-1} \|_{L^2(\Omega;\R^d)} \big)^2. 
  \end{align*}
  Together with the \emph{a priori} bounds from Lemma~\ref{lem:apriori}
  this shows \eqref{eq6:term3}.

  It remains to prove the estimates \eqref{eq6:Greg1} and \eqref{eq6:Greg2}.
  For \eqref{eq6:Greg1} we first apply the Burkholder--Davis--Gundy-type
  inequality from Lemma~\ref{lem:hoelder} with constant $C_\rho$ 
  and obtain for every 
  $n \in \{1,\ldots,N\}$ and $t \in (t_{n-1},t_n]$ that
  \begin{align*}
    \int_{t_{n-1}}^{t} \E \big[ | G(t) - G(s) |^{\rho} \big] \diff{s}
    &\le C_\rho^{\rho} \int_{t_{n-1}}^{t} (t - s)^{\frac{\rho - 2}{2}}
    \int_s^{t} \E \big[ |g(X(\tau))|^\rho \big] \diff{\tau} \diff{s}\\
    &\le \frac{2^{\rho} }{\rho} C_{\rho}^{\rho} L_g^\rho k^{\frac{\rho}{2}}
    \int_{t_{n-1}}^{t} \big( 1 + \E \big[ | X(\tau)|^\rho \big] \big) \diff{\tau},     
  \end{align*}
  where we also made use of the linear growth bound 
  \eqref{eq4:lingrowthbg} 
  in the last step. This proves \eqref{eq6:Greg1}.
  The bound in \eqref{eq6:Greg2} can be shown by analogous
  arguments.
\end{proof}

The next lemma generalizes an important estimate from the proof of
Theorem~\ref{thm:conv_Hoelder} to the multi-valued setting. In particular, we
refer to Lemma~\ref{lem:err_rep2} and \eqref{eq5:H1}.

\begin{lemma}
  \label{lem:coerc}
  Let Assumptions~\ref{ass:f} to \ref{ass:X0} be satisfied.
  For every step size $k = \frac{T}{N}$, $N \in \N$, with
  $5 L_b k \in [0,1)$, let the families $(X^n)_{n\in \{0,\dots,N\}}$ and
  $(\eta^n)_{n\in \{0,\dots,N\}}$ of random variables be as stated in 
  Theorem~\ref{thm:existence_multi}.
  Then there exists $K_{\delta \eta} \in (0,\infty)$ independent of the step
  size $k$ such that
  \begin{align*}  
    0\le&k \sum_{i=1}^{N} \E \big[ \inner{\eta^i - \eta^{i-1}}{X^i - 
    X^{i-1}}\big] \le K_{\delta\eta} k^{\frac{1}{2}}.  
  \end{align*}
\end{lemma}

\begin{proof}
  The nonnegativity follows immediately from the monotonicity of $f$.
  To prove the second inequality, we insert the scheme 
  \eqref{eq6:multiEuler} 
  and obtain
  \begin{align}
    \notag &k \sum_{i=1}^{N} \E \big[ \inner{\eta^i - \eta^{i-1}}{X^i - 
    X^{i-1}}\big]\\
     &\quad = k \sum_{i=1}^{N} \E \big[ \inner{\eta^i - \eta^{i-1}}{k
     (b(X^i) - \eta^i) + g(X^{i-1}) \Delta W^i}\big] \notag\\
    \label{eq6:proofestimate1}
    &\quad = - k^2 \sum_{i=1}^{N}
    \E \big[ \inner{\eta^i - \eta^{i-1}}{\eta^i}\big]\\
    \label{eq6:proofestimate2}& \qquad + k \sum_{i=1}^{N} \E \big[ 
    \inner{\eta^i - \eta^{i-1} }{kb(X^i) + g(X^{i-1}) \Delta W^i}\big].
  \end{align}
  For \eqref{eq6:proofestimate1} we obtain 
  \begin{align*}
    - k^2 \sum_{i=1}^{N} \E \big[ \inner{\eta^i - \eta^{i-1}}{\eta^i}\big]
    &=- \frac{k^2}{2} \sum_{i=1}^{N} \E \big[ |\eta^i |^2 - |\eta^{i-1}|^2 
    + |\eta^i - \eta^{i-1} |^2\big] \\
    &\leq - \frac{k^2}{2} \big(\E \big[ |\eta^N |^2\big] - \E \big[ 
    |\eta^{0}|^2\big]\big)
    \leq \frac{k^2}{2} \E \big[ |\eta^0|^2\big]
  \end{align*}
  because of the telescopic structure. Furthermore, it follows
  from Assumptions~\ref{ass:f} and \ref{ass:X0} that 
  \begin{align*}
    \big(\E \big[ |\eta^0|^2\big]\big)^{\frac{1}{2}} 
    \le \beta \big( 1 + \big(\E \big[ |X_0|^{2p-2} \big]\big)^{\frac{1}{2}} \big) < 
    \infty.
  \end{align*}
  For \eqref{eq6:proofestimate2} we apply H\"older's inequality with $\rho = 
  \max(2,p)$ and $\frac{1}{\rho} + \frac{1}{\rho'} = 1$ 
  to obtain
  \begin{align*}
    &k \sum_{i=1}^{N} \E [\inner{\eta^i - \eta^{i-1} }{kb(X^i) 
    + g(X^{i-1}) \Delta W^i}]\\
    &\quad \leq k \sum_{i=1}^{N} \big(\E [ |\eta^i - \eta^{i-1} |^{\rho'}
    ]\big)^{\frac{1}{\rho'}} 
    \big(\E[ |k b(X^i) + g(X^{i-1}) \Delta W^i|^{\rho}]\big)^{\frac{1}{\rho}} \\
    &\quad \leq \Big(k \sum_{i=1}^{N} \E [ |\eta^i - \eta^{i-1}|^{\rho'}
    ]\Big)^{\frac{1}{\rho'}} 
    \Big(k \sum_{i=1}^{N} \E[ |k b(X^i) + g(X^{i-1}) \Delta
    W^i|^{\rho}]\Big)^{\frac{1}{\rho}}.  
  \end{align*}
  Then, from applications of the triangle inequality and
  Lemma~\ref{lem:apriori}, we get
  \begin{align*}
    \Big(k \sum_{i=1}^{N} \E [ |\eta^i - \eta^{i-1}|^{\rho'}
    ]\Big)^{\frac{1}{\rho'}} 
    &\le \Big(k \sum_{i=1}^{N} \E [ |\eta^i |^{\rho'}
    ]\Big)^{\frac{1}{\rho'}} 
    + \Big(k \sum_{i=1}^{N} \E [ |\eta^{i-1}|^{\rho'}
    ]\Big)^{\frac{1}{\rho'}} \\
    &\le K_\eta^{\frac{1}{\rho'}} + \big( K_\eta + \E[ |\eta^0|^{\rho'} ]
    \big)^{\frac{1}{\rho'}}
    \le 2 K_\eta^{\frac{1}{\rho'}} + \big( \E[ |\eta^0|^{\rho'} ]    
    \big)^{\frac{1}{\rho'}}.
  \end{align*}
  We apply the polynomial growth bound satisfied 
  by $f$ and see that, for $p \in [2,\infty)$,
  \begin{align*}
    \big( \E[ |\eta^0|^{\rho'} ]\big)^{\frac{1}{\rho'}}
    =\big( \E[ |\eta^0|^{q} ]\big)^{\frac{1}{q}}
    \le \beta( 1 + \| X_0 \|_{L^p(\Omega;\R^d)}^{p-1})
  \end{align*}
  is fulfilled, while for $p \in [1,2)$ we have
  \begin{align*}
    \big( \E[ |\eta^0|^{\rho'} ]\big)^{\frac{1}{\rho'}}
    = \big(\E[ |\eta^0|^2 ] \big)^{\frac{1}{2}}
    &\le \beta\big( 1 + \big(\E[ |X_0|^{2p-2} ]\big)^{\frac{1}{2}} \big)
    = \beta\big( 1 + \| X_0 \|_{L^{2p-2}(\Omega;\R^d)}^{p-1}  \big).
  \end{align*}
  In both cases the appearing terms are finite because of 
  Assumption~\ref{ass:X0}.
  Moreover, a further application of the triangle inequality yields
  \begin{align*}
    &\Big(k \sum_{i=1}^{N} \E[ |k b(X^i) + g(X^{i-1}) \Delta
    W^i|^{\rho}]\Big)^{\frac{1}{\rho}}\\
    &\quad \le \Big(k \sum_{i=1}^{N} \E[ |k b(X^i)|^{\rho}]\Big)^{\frac{1}{\rho}}
    + \Big(k \sum_{i=1}^{N} \E[ | g(X^{i-1}) \Delta
    W^i|^{\rho}]\Big)^{\frac{1}{\rho}}.
  \end{align*}
  Due to the linear growth bound \eqref{eq4:lingrowthbg} on $b$ and
  the \emph{a priori}
  bound \eqref{eq4:aprioriX}, it then follows that
  \begin{align*}
    \Big(k \sum_{i=1}^{N} \E[ |k b(X^i)|^{\rho}]\Big)^{\frac{1}{\rho}}
    &\le L_b k \Big( k \sum_{i=1}^{N} \E\big[ \big(1 + | X^i|
    \big)^{\rho}  \big]\Big)^{\frac{1}{\rho}}\\
    &\le L_b k \Big( T^{\frac{1}{\rho}} + \big( \max\big(\frac{1}{2 \mu},T\big) 
    K_X \big)^{\frac{1}{\rho}} \Big).
  \end{align*}
  By application of Lemma~\ref{lem:hoelder} with constant $C_{\rho}$, we obtain
  \begin{align*}
    \E[ |g(X^{i-1}) \Delta W^i|^{\rho}]
    = \E \Big[ \Big|\int_{t_{i-1}}^{t_i} g(X^{i-1}) \diff{W(s)}\Big|^{\rho} \Big]
    \leq C_{\rho}^{\rho} k^{\frac{\rho}{2}}
    \E \big[ |g(X^{i-1})|^{\rho} \big].
  \end{align*}
  Together with the linear growth bound \eqref{eq4:lingrowthbg} on
  $g$, this shows that 
  \begin{align*}
    \Big( k \sum_{i=1}^{N} \E[ |g(X^{i-1}) \Delta 
    W^i|^{\rho}]\Big)^{\frac{1}{\rho}}
    &\leq C_{\rho} k^{\frac{1}{2}} \Big( k \sum_{i=1}^{N} \E \big[
    |g(X^{i-1})|^{\rho} \big] \Big)^{\frac{1}{\rho}}\\
    &\le C_{\rho} L_g k^{\frac{1}{2}} \Big( T^{\frac{1}{\rho}} + \big(
    \max\big(\frac{1}{2 \mu},T\big) K_X \big)^{\frac{1}{\rho}} \Big).
  \end{align*}
  Putting the estimates together proves the desired bound.
\end{proof}

We are now prepared to state and prove the main result of this section.
While the main ingredients of the proof
still consist of techniques introduced in \cite[Sect.~4]{nochetto2000} for
deterministic problems, the proof is somewhat more technical than the proof of
Theorem~\ref{thm:conv_Hoelder}. In particular, 
due to the presence of Lipschitz perturbations in
the general problem \eqref{eq1:SDE} it is no longer possible to avoid an
application of a Gronwall lemma. Moreover, 
as in \cite[Sect.~4]{nochetto2000} we impose the following additional
assumption on the multi-valued mapping $f$.

\begin{assumption}
  \label{ass:NochettoIneq}
  There exists $\gamma \in (0,\infty)$ such that, for every $v,w, z
  \in D(f)$, $f_v \in f(v)$, $f_w \in f(w)$, and $f_z \in f(z)$, 
  \begin{align*}
    \inner{f_v - f_z}{z-w} \leq \gamma \inner{f_v - f_w}{v-w}.
  \end{align*}
\end{assumption}

In Lemma~\ref{lem:NochettoIneq}, we already proved that, if $f$ is the
subdifferential of a convex potential, then Assumption~\ref{ass:NochettoIneq}
is satisfied with $\gamma = 1$. For a further example, we refer to
Section~\ref{sec:examples}. 

\begin{theorem}
  \label{thm:conv}
  Let Assumptions~\ref{ass:f} -- \ref{ass:X0} and
  Assumption~\ref{ass:NochettoIneq} be satisfied. 
  Let the step size $k = \frac{T}{N}$, $N \in \N$, be such that $8 L_b k \in 
  [0,1)$. Then there exists a constant $C \in (0,\infty)$ independent 
  of $k$ such 
  that
  \begin{align*}
    \max_{t \in [0,T]} \| X(t) - \X(t) \|_{L^2(\Omega;\R^d)}
    \le C k^{\frac{1}{4}}.
  \end{align*}
\end{theorem}

\begin{proof}
  Let us first introduce some additional notation.
  We will denote the error between the exact solution $X$ to \eqref{eq1:SDE} 
  and the numerical approximation $\X$ defined in \eqref{eq6:Xlin}
  by $E(t):= X(t) - \X(t)$, $t \in [0,T]$. 
  Furthermore, it will be convenient to split the error into two parts 
  \begin{align*} 
     E(t) = E_1(t) + E_2(t), \quad t \in [0,T],
  \end{align*}
  where we define
  \begin{align}
    \label{eq4:errorE1}
    E_1(t) &:= \int_{0}^{t} \big( \overline{\Eta}(s) -  \eta(s) \big)
    \diff{s} 
    + \int_{0}^{t}  \big( b(X(s)) - b(\overline{\X}(s)) \big)
    \diff{s},\\
    \label{eq4:errorE2}
    E_2(t) &:=  G(t) - \G(t)
  \end{align}
  $\P$-almost surely for every $t \in (0,T]$.
  We expand the square of the norm of $E$ as
  \begin{align} 
    \label{eq4:normError}
    |E(t)|^2  = |E_1(t)|^2 + 2 \inner{E_1(t)}{E_2(t)} + |E_2(t)|^2, \quad t \in
    [0,T].
  \end{align}
  In order to estimate the terms on the right-hand side
  of \eqref{eq4:normError}, we first observe in \eqref{eq4:errorE1} that
  $E_1$ has absolutely continuous sample paths with $E_1(0)=0$.
  Hence we have $\frac{1}{2} \frac{\mathrm{d}}{\diff{t}}|E_1(t)|^2 
  = \langle \dot{E}_1(t), E_1(t) \rangle$ for almost every $t \in [0,T]$. 
  Therefore, after integrating from $0$ to $t \in (0,T]$ we get
  \begin{align} 
    \label{eq4:normE1}
    \begin{split}
      \frac{1}{2} |E_1(t)|^2 
    = \int_{0}^{t} \inner{\dot{E}_1(s)}{E_1(s)} \diff{s}
    = \int_{0}^{t} \inner{\dot{E}_1(s)}{E(s)} \diff{s}
    - \int_{0}^{t} \inner{\dot{E}_1(s)}{E_2(s)} \diff{s}.
    \end{split}
  \end{align}
  Furthermore, we also have
  \begin{align}
    \label{eq6:normE1E2}
    \inner{E_1(t)}{E_2(t)}
    =  \Big\langle \int_0^t \dot{E}_1(s) \diff{s}, E_2(t) \Big\rangle
    =  \int_0^t \inner{\dot{E}_1(s)}{E_2(t)} \diff{s}.
  \end{align}
  Thus, after combining \eqref{eq4:normE1} and \eqref{eq6:normE1E2} we obtain
  \begin{align}
    \label{eq6:2int}
    \frac{1}{2} |E_1(t)|^2 +  \inner{E_1(t)}{E_2(t)}
    &=  \int_0^t \inner{\dot{E}_1(s)}{E(s)} \diff{s}
    +  \int_0^t \inner{\dot{E}_1(s)}{E_2(t)-E_2(s)} \diff{s}.
  \end{align}
  For the first integral on the right-hand side of 
  \eqref{eq6:2int} we insert the derivative of $E_1$ and the definition of the 
  error process $E$. This yields, for almost every $s \in (0, T]$,
  \begin{align*}
    \inner{\dot{E}_1(s)}{E(s)}
    =  \inner{\overline{\Eta}(s) - \eta(s)}{X(s) - \X(s)}
    + \inner{b(X(s)) - b(\overline{\X}(s))}{X(s) - \X(s)}.
  \end{align*}
  After recalling the definition of $\X$ we use Assumptions~\ref{ass:f} 
  and \ref{ass:NochettoIneq}. Then, for almost every $s \in (t_{n-1},t_n]$ 
  and all $n \in \{1,\ldots,N\}$, we get
  \begin{align*}
    &\inner{\overline{\Eta}(s) - \eta(s)}{X(s) - \X(s)}\\
    &\quad = \frac{t_n - s}{k}\inner{\eta^n - \eta(s) }{X(s) -
      X^{n-1}}
    + \frac{s - t_{n-1}}{k} \inner{\eta^n - \eta(s) }
    {X(s) - X^n}\\
    &\quad \leq \gamma \frac{t_n - s}{k} \inner{\eta^n - \eta^{n-1} 
    }{X^n - X^{n-1}}
    - \frac{s - t_{n-1}}{k} \inner{\eta(s)- \eta^n}{X(s) - X^n}\\
    &\quad \leq \gamma \frac{t_n - s}{k} \inner{ \eta^n - \eta^{n-1} }{X^n -
    X^{n-1}},
  \end{align*}
  where the second term in the last step is non-positive due to the
  monotonicity of $f$ (cf.~Definition~\ref{def:mono}). Moreover, because of the
  Lipschitz continuity of $b$, we have for almost every $s \in (0,T]$ that
  \begin{align*}
    &\inner{b(X(s)) - b(\overline{\X}(s))}{X(s) - \X(s)}\\
    &\quad = \inner{b(X(s)) - b(\X(s))}{X(s) - \X(s)}
    + \inner{b(\X(s)) - b(\overline{\X}(s))}{X(s) - \X(s)}\\
    &\quad \leq L_b |E(s)|^2
    + L_b |\X(s) - \overline{\X}(s)| |E(s)|
    \leq \frac{3}{2} L_b |E(s)|^2
    + \frac{L_b}{2} |\X(s) - \overline{\X}(s)|^2,
  \end{align*}
  where we also made use of Young's inequality.
  In addition, for every $n \in \{1,\ldots,N\}$ and 
  $s \in (t_{n-1},t_n]$, we have that
  \begin{align*}
    \X(s) - \overline{\X}(s)
    = \frac{s-t_{n-1}}{k} X^n + \frac{t_n-s}{k} X^{n-1} - X^n
    = - \frac{t_n-s}{k} \big(X^n - X^{n-1}\big).
  \end{align*}
  Therefore,
  \begin{align*}
    \inner{b(X(s)) - b(\overline{\X}(s))}{X(s) - \X(s)}
    \leq \frac{3}{2} L_b |E(s)|^2
    + \frac{L_b(t_n-s)^2}{2 k^2} |X^n - X^{n-1}|^2.
  \end{align*}
  Altogether, for every $t \in (t_{n-1}, t_n]$ and $n \in
  \{1,\ldots,N\}$, we have shown that
  \begin{align*}
    \int_{t_{n-1}}^t \inner{\dot{E}_1(s)}{E(s)} \diff{s}
    &\le \frac{\gamma}{2} k
    \inner{ \eta^n - \eta^{n-1} }{X^n - X^{n-1}}\\
    &\quad + \frac{3}{2} L_b \int_{t_{n-1}}^{t} | E(s)|^2 \diff{s}
    + \frac{L_b}{6} k |X^n - X^{n-1}|^2,
  \end{align*}
  where we also inserted that
  $\int_{t_{n-1}}^{t} (t_n - s) \diff{s} 
  \le \int_{t_{n-1}}^{t_n} (t_n - s) \diff{s} 
  = \frac{1}{2} k^2$ as well as $\int_{t_{n-1}}^{t} (t_n - s)^2 \diff{s} 
  \le \frac{1}{3} k^3$. It follows that, for every
  $n \in \{1,\ldots,N\}$ and $t \in (t_{n-1},t_n]$,
  \begin{align*}
      &\int_0^t\inner{\dot{E}_1(s)}{E(s)} \diff{s}
      = \sum_{i = 1}^{n-1} \int_{t_{i-1}}^{t_i} 
      \inner{\dot{E}_1(s)}{E(s)} \diff{s}
      + \int_{t_{n-1}}^t \inner{\dot{E}_1(s)}{E(s)} \diff{s} 
      \notag\\
      &\quad \le \frac{\gamma}{2} k \sum_{i = 1}^n
      \inner{ \eta^i - \eta^{i-1} }{X^i - X^{i-1}}
      + \frac{L_b}{6} k \sum_{i = 1}^n
      |X^i - X^{i-1}|^2 \notag
      + \frac{3}{2} L_b \int_0^t |E(s)|^2 \diff{s}.
  \end{align*}
  Hence, together with Lemma~\ref{lem:apriori} and Lemma~\ref{lem:coerc}
  this shows that 
  \begin{align}
    \label{eq6:term1}
    \int_0^t \E\big[\inner{\dot{E}_1(s)}{E(s)}\big] \diff{s}
    &\le \frac{\gamma}{2} K_{\delta \eta} k^{\frac{1}{2}}
    + \frac{L_b}{3} K_X k + \frac{3}{2} L_b \int_0^t \E\big[|E(s)|^2\big] \diff{s}.
  \end{align}
  Next, we give an estimate for the second 
  integral on the right-hand side of \eqref{eq6:2int}. 
  For every $n \in \{1,\ldots,N\}$ and $t \in (t_{n-1},t_n]$ we decompose the
  integral as follows
  \begin{align}
    \label{eq6:term2}
    \begin{split}
      \int_{0}^{t} \inner{\dot{E}_1(s)}{E_2(t) - E_2(s)} \diff{s}
      &= \sum_{i = 1}^{n-1} \int_{t_{i-1}}^{t_i} 
      \inner{\dot{E}_1(s)}{E_2(t) - E_2(s)} \diff{s}\\
      &\quad + \int_{t_{n-1}}^t \inner{\dot{E}_1(s)}{E_2(t) - E_2(s)} \diff{s}.
    \end{split}
  \end{align}
  For every $i \in \{1,\ldots,n-1\}$ we then add
  and subtract $E_2(t_i)$ in the second slot of the inner product in the first
  term on the right-hand side of \eqref{eq6:term2}. This gives
  \begin{align*}
    \int_{t_{i-1}}^{t_i} \inner{\dot{E}_1(s)}{E_2(t) - E_2(s)} \diff{s}
    &= \int_{t_{i-1}}^{t_i} 
    \inner{\dot{E}_1(s)}{E_2(t) - E_2(t_{i})} \diff{s}\\
    &\quad + \int_{t_{i-1}}^{t_i} 
    \inner{\dot{E}_1(s)}{E_2(t_{i}) - E_2(s)} \diff{s}.
  \end{align*}
  After inserting the definition of $E_2$ from
  \eqref{eq4:errorE2} the first integral is then equal to
  \begin{align*}
    &\int_{t_{i-1}}^{t_i} \inner{\dot{E}_1(s)}{E_2(t) - E_2(t_{i})} \diff{s}
    = \Big\langle \int_{t_{i-1}}^{t_i} \dot{E}_1(s) \diff{s}, E_2(t) -
    E_2(t_{i}) \Big\rangle\\
    &\quad = \inner{E_1(t_i) - E_1(t_{i-1})}{E_2(t) - E_2(t_{i})}\\
    &\quad = \inner{E_1(t_i) - E_1(t_{i-1})}{G(t) -
    \G(t) - ( G(t_{i}) - \G(t_i)) } \\
    &\quad = \Big\langle E_1(t_i) - E_1(t_{i-1}),
    \int_{t_i}^{t} g(X(s)) \diff{W(s)} \Big\rangle\\
    &\qquad - \Big\langle E_1(t_i) - E_1(t_{i-1}),
    \int_{t_i}^{t_{n-1}} g(\underline{\X}(s)) \diff{W(s)} + 
    \frac{t - t_{n-1}}{k} g(X^{n-1}) \Delta W^n \Big\rangle
  \end{align*}
  for all $i, n \in \{1,\ldots,N\}$, $i < n$, and $t \in (t_{n-1},t_n]$.
  Since $E_1(t_i) - E_1(t_{i-1}) = E(t_i) - E(t_{i-1})
  - (E_2(t_i) - E_2(t_{i-1}))$ is square-integrable
  and $\F_{t_i}$-measurable it therefore follows that
  \begin{align*}
    \E \Big[ \int_{t_{i-1}}^{t_i} \inner{\dot{E}_1(s)}{E_2(t) - E_2(t_{i})}
    \diff{s} \Big] = 0
  \end{align*}
  for all $n \in \{1,\ldots,N\}$, $t \in (t_{n-1},t_n]$ and $t_i < t$.
  Hence, after taking expectations in \eqref{eq6:term2} we arrive at
  \begin{align*}
    &\E \Big[ \int_{0}^{t} \inner{\dot{E}_1(s)}{E_2(t) - E_2(s)} \diff{s}
    \Big]\\ 
    &= \sum_{i = 1}^{n-1} \E \Big[ \int_{t_{i-1}}^{t_i} 
    \inner{\dot{E}_1(s)}{E_2(t_i) - E_2(s)} \diff{s} \Big]
    + \E \Big[ \int_{t_{n-1}}^t \inner{\dot{E}_1(s)}{E_2(t) - E_2(s)} \diff{s}
    \Big]\\
    &\leq \sum_{i = 1}^{n} \E \Big[ \int_{t_{i-1}}^{t_i} 
    |\dot{E}_1(s)| |E_2(t_i) - E_2(s)| \diff{s} \Big].
  \end{align*}
  Inserting the definitions \eqref{eq4:errorE1} and
  \eqref{eq4:errorE2} of $E_1$ and $E_2$ and applying H\"older's inequality 
  with $\rho = \max(2,p)$ and
  $\frac{1}{\rho}+\frac{1}{\rho'}=1$, we get
  \begin{align*}
    &\E \Big[ \int_{0}^{t} \inner{\dot{E}_1(s)}{E_2(t) - E_2(s)} \diff{s}
    \Big]\\ 
    &\leq \sum_{i = 1}^{n} \int_{t_{i-1}}^{t_i} \E \big[
    \big(|\eta^i - \eta(s)| + |b(X(s)) - b(X^i)| \big)\\
    &\qquad \times \big(| G(t_i) - G(s)| + |\G(t_i) - \G(s)| \big) \big]
    \diff{s}\\
    &\leq \sum_{i = 1}^{n}
    \Big( \int_{t_{i-1}}^{t_i} \E \big[
    |\eta^i - \eta(s)|^{\rho'} \big] \diff{s} \Big)^{\frac{1}{\rho'}}
    \Big( \int_{t_{i-1}}^{t_i}
    \E \big[ | G(t_i) - G(s)|^\rho \big] \diff{s} \Big)^{\frac{1}{\rho}}\\
    &\quad + \sum_{i = 1}^{n}
    \Big( \int_{t_{i-1}}^{t_i} \E \big[
    |\eta^i - \eta(s)|^{\rho'} \big] \diff{s} \Big)^{\frac{1}{\rho'}}
    \Big( \int_{t_{i-1}}^{t_i}
    \E \big[ | \G(t_i) - \G(s)|^\rho \big] \diff{s} \Big)^{\frac{1}{\rho}}\\
    &\quad + \sum_{i = 1}^{n}
    \Big( \int_{t_{i-1}}^{t_i} \E \big[
    |b(X(s)) - b(X^i)|^{\rho'} \big] \diff{s} \Big)^{\frac{1}{\rho'}}
    \Big( \int_{t_{i-1}}^{t_i}
    \E \big[ | G(t_i) - G(s)|^\rho \big] \diff{s} \Big)^{\frac{1}{\rho}}\\
    &\quad + \sum_{i = 1}^{n}
    \Big( \int_{t_{i-1}}^{t_i} \E \big[
    |b(X(s)) - b(X^i)|^{\rho'} \big] \diff{s} \Big)^{\frac{1}{\rho'}}
    \Big( \int_{t_{i-1}}^{t_i}
    \E \big[ | \G(t_i) - \G(s)|^\rho \big] \diff{s} \Big)^{\frac{1}{\rho}}\\
    &=:  \Gamma_1 + \Gamma_2 + \Gamma_3 + \Gamma_4.
  \end{align*}
  In the following, we will estimate $\Gamma_1$, $\Gamma_2$, $\Gamma_3$, 
  and $\Gamma_4$ separately. For $\Gamma_1$ we obtain after an 
  application of H\"older's inequality for sums that
  \begin{align*}
    \Gamma_1 
    &\leq 
    \Big( \sum_{i = 1}^{n} \int_{t_{i-1}}^{t_i} \E \big[
    |\eta^i - \eta(s)|^{\rho'} \big] \diff{s} \Big)^{\frac{1}{\rho'}}
    \Big( \sum_{i = 1}^{n} \int_{t_{i-1}}^{t_i}
    \E \big[ | G(t_i) - G(s)|^\rho \big] \diff{s} \Big)^{\frac{1}{\rho}}\\
    &\leq 
    \Big(\Big( k \sum_{i = 1}^{n}
    \E \big[ |\eta^i|^{\rho'} \big] \Big)^{\frac{1}{\rho'}}
    + \Big( \int_{0}^{t_n} \E \big[
    |\eta(s)|^{\rho'} \big] \diff{s}
    \Big)^{\frac{1}{\rho'}}\Big)\\
    &\qquad \times
    \Big( \sum_{i = 1}^{n} \int_{t_{i-1}}^{t_i}
    \E \big[ | G(t_i) - G(s)|^\rho \big] \diff{s} \Big)^{\frac{1}{\rho}}. 
  \end{align*}
  If $p \in [2,\infty)$ then $\rho = p$ and $\rho' = q$. In this case all 
  integrals appearing are finite due to the bounds in 
  Theorem~\ref{thm:exact} and Lemma~\ref{lem:apriori}.
  Moreover, if $p \in (1,2)$ then $\rho = \rho' = 2 < q$. Then it follows from
  further applications of H\"older's inequality and Jensen's inequality that
  \begin{align*}
    k \sum_{i = 1}^{n} \E \big[ |\eta^i|^{2} \big]     
    \le T^{\frac{q-2}{2}} \Big( k \sum_{i = 1}^{n}
    \E \big[ |\eta^i|^{q} \big] \Big)^{\frac{2}{q}} 
  \end{align*}
  as well as
  \begin{align*}
    \int_{0}^{t_n} \E \big[ |\eta(s)|^{2} \big] \diff{s}
    \le T^{\frac{q-2}{2}} \Big(
    \int_{0}^{t_n} \E \big[ |\eta(s)|^{q} \big] \diff{s}
    \Big)^{\frac{2}{q}}.
  \end{align*}
  Hence, we arrive at the same conclusion. If $p = 1$ then the processes 
  $(\eta(t))_{t \in [0,T]}$ and $(\eta^n)_{n \in \{1,\ldots,N\}}$ are globally 
  bounded due to the bound on $f$ in Assumption~\ref{ass:f}. Using 
  Lemma~\ref{lem:Gerror} we see that
  \begin{align*}
    \Big( \sum_{i = 1}^{n} \int_{t_{i-1}}^{t_i}
    \E \big[ | G(t_i) - G(s)|^\rho \big] \diff{s} \Big)^{\frac{1}{\rho}}
    \leq K_{\rho} k^{\frac{1}{2}}
    \Big( \int_0^{t_n} \big( 1 + \E \big[ | X(s)|^\rho \big] \big) \diff{s}
    \Big)^{\frac{1}{\rho}}.
  \end{align*}
  Altogether, this yields
  \begin{align*}
    \Gamma_1 \le C_{\Gamma_1} k^{\frac{1}{2}}
  \end{align*}
  for a suitable constant $C_{\Gamma} \in (0,\infty)$, which is independent 
  of $k$.
  To estimate $\Gamma_2$ we argue analogously as in the case for $\Gamma_1$ to obtain 
  that
  \begin{align*}
    \Gamma_2 
    &\leq 
    \Big(\Big( k \sum_{i = 1}^{n}
    \E \big[ |\eta^i|^{\rho'} \big] \Big)^{\frac{1}{\rho'}}
    + \Big( \int_{0}^{t} \E \big[
    |\eta(s)|^{\rho'} \big] \diff{s}
    \Big)^{\frac{1}{\rho'}}\Big)\\
    &\qquad \times
    \Big( \sum_{i = 1}^{n} \int_{t_{i-1}}^{t_i}
    \E \big[ | \G(t_i) - \G(s)|^\rho \big] \diff{s} \Big)^{\frac{1}{\rho}}. 
  \end{align*}
  The first factor is bounded as we saw in the case for $\Gamma_1$. Furthermore, using 
  Lemma~\ref{lem:Gerror}, we have that
  \begin{align*}
    \Big( \sum_{i = 1}^{n} \int_{t_{i-1}}^{t_i}
    \E \big[ | \G(t_i) - \G(s)|^\rho \big] \diff{s} \Big)^{\frac{1}{\rho}}
    \leq K_{\rho} k^{\frac{1}{2}}
    \Big( k \sum_{i =1}^{n} \big( 1 + \E \big[ | X^{i-1}|^\rho \big] \big) \diff{s}
    \Big)^{\frac{1}{\rho}}.
  \end{align*}
  Due to the \emph{a priori} bound \eqref{eq4:aprioriX}, it follows that there exists a 
  constant $C_{\Gamma_2} \in (0,\infty)$, which does not depend on $k$ 
  such that
  \begin{align*}
    \Gamma_2 \le C_{\Gamma_2} k^{\frac{1}{2}}.
  \end{align*}
  The estimates $\Gamma_3$ and $\Gamma_4$ follow analogously with the 
  only new term that appears is of the form
  \begin{align*}
    &\Big( \sum_{i = 1}^{n} \int_{t_{i-1}}^{t_i} \E \big[
    |b(X(s)) - b(X^i)|^{\rho'} \big] \diff{s} \Big)^{\frac{1}{\rho'}}\\
    &\leq L_b \Big( \sum_{i = 1}^{n} \int_{t_{i-1}}^{t_i} \E \big[
    |X(s) - X^i|^{\rho'} \big] \diff{s} \Big)^{\frac{1}{\rho'}}\\
    &\leq L_b \Big( \int_{0}^{t_n} \E \big[ |X(s)|^{\rho'} \big] \diff{s} 
    \Big)^{\frac{1}{\rho'}}
    + L_b \Big( k \sum_{i = 1}^{n} \E \big[|X^i|^{\rho'} \big] \diff{s} 
    \Big)^{\frac{1}{\rho'}},
  \end{align*}
  which is bounded due to Theorem~\ref{thm:exact} and the \emph{a priori} bound 
  \eqref{eq4:aprioriX}. Therefore, there exist constants $C_{\Gamma_3}, 
  C_{\Gamma_4} \in (0,\infty)$ such that
  \begin{align*}
    \Gamma_3 \le C_{\Gamma_3} k^{\frac{1}{2}}
    \quad \text{and} \quad
    \Gamma_4 \le C_{\Gamma_4} k^{\frac{1}{2}}.
  \end{align*}
  Hence, we obtain 
  \begin{align}
    \label{eq6:term2c}
    \E \Big[ \int_{0}^{t} \inner{\dot{E}_1(s)}{E_2(t) - E_2(s)} \diff{s}
    \Big] \le (C_{\Gamma_1} + C_{\Gamma_2} + C_{\Gamma_3} + 
    C_{\Gamma_4}) k^{\frac{1}{2}}
    =: C_{\Gamma} k^{\frac{1}{2}} .
  \end{align}
  After taking expectations in \eqref{eq4:normError}
  and inserting \eqref{eq6:2int}, \eqref{eq6:term1}, \eqref{eq6:term2c} as well
  as \eqref{eq6:term3} from Lemma~\ref{lem:Gerror}, 
  we obtain for every $t \in (0,T]$ that
  \begin{align*}
    \E\big[ |E(t)|^2 \big] 
    &\le \gamma K_{\delta \eta} k^{\frac{1}{2}}
    + \frac{2 L_b}{3} K_X k + 2 C_\Gamma k^{\frac{1}{2}} 
    + K_G k + \big( 3 L_b + 2 L_g^2 \big) \int_0^t |E(s)|^2 \diff{s}.
  \end{align*} 
  The assertion then follows from an application of Gronwall's lemma, see 
  for example, \cite[Appendix~B]{evans1998}.
\end{proof}

\begin{remark}\label{rem:noteta}
  Up to this point, we only proved convergence for $X$ but not for $\eta$.
  However, from the existence of $X^n$ we also obtain that 
  \begin{align*}
    k \eta^n = - (X^n - X^{n-1}) + k b(X^n) + g(X^{n-1}) \Delta W^n 
    \text{ a.s. in } \Omega.
  \end{align*}
  Analogously, we can write for the exact solution $\eta$ that 
  \begin{align*}
    \int_{0}^{t} \eta(s) \diff{s} = - X(t) + X_0 + \int_0^t b(X(s)) \diff{s} 
    + \int_{0}^{t} g(X(s)) \diff{W(s)}.
  \end{align*}
  Therefore, from the convergence of $\X$ to $X$ and the Lipschitz continuity
  of $b$ and $g$ we also obtain the estimate 
  \begin{align*}
    \Big\| \int_{0}^{t_n} \eta(s) \diff{s} - k\sum_{j = 1}^n \eta^j 
    \Big\|_{L^2(\Omega;\R^d)} \le C k^{\frac{1}{4}}
  \end{align*}
  for every $n \in \{1,\ldots,N\}$. 
\end{remark}

\section{Examples}
\label{sec:examples}

\subsection{Discontinuous drift coefficient}
In this example, we show that Assumption~\ref{ass:f} 
includes overdamped Langevin-type equations with a possibly 
discontinuous drift $f$. We consider the convex,
nonnegative, yet not continuously differentiable function $\Phi(x) := |x|$, $x
\in \R$, which has a multi-valued subdifferential $f \colon \R \to 2^{\R}$
defined by 
\begin{align*}
  f(x) :=
  \begin{cases}
    \{1\},& \text{ if } x > 0,\\
    [-1,1],& \text{ if } x = 0,\\
    \{-1\},& \text{ if } x < 0.
  \end{cases}
\end{align*}
This mapping fulfills Assumption~\ref{ass:f} for $p=1$. To be more precise, $f$
is a monotone function and there exists no proper  
monotone extension of its graph. In fact, the subdifferential of any proper,
lower semi-continuous and convex function is a maximal monotone mapping by a 
well-known theorem of Rockafellar, cf. \cite[Cor.~31.5.2]{rockafellar1997} or 
\cite[Satz~3.23]{ruzicka2004}. 

Furthermore, we notice that $f_x x = \text{sgn} (x) x = |x|$ as well as $|f_x | 
\leq 1$ for every $x \in \R$ and $f_x \in f(x)$. This shows that $f$ fulfills all 
the conditions of Assumption~\ref{ass:f}. It remains to verify 
Assumption~\ref{ass:NochettoIneq}. Since $f$ is the subdifferential of $\Phi$ 
the variational inequality \eqref{eq1:varineq} is still satisfied in the sense
that 
\begin{align*}
  f_x(y-x) \leq  \Phi(y) - \Phi(x)
\end{align*}
for all $x,y \in \R$ and $f_x \in f(x)$. Following the same steps as in the
proof of Lemma~\ref{lem:NochettoIneq} but replacing $f(v)$, $f(w)$, 
and $f(z)$ by arbitrary elements $f_v \in f(v)$, $f_w \in f(w)$, 
and $f_z \in f(z)$, respectively, shows that 
Assumption~\ref{ass:NochettoIneq} is fulfilled.
Therefore, the backward Euler--Maruyama method \eqref{eq1:multiEuler} 
is well-defined and yields an approximation of the exact solution $X$ of 
\begin{align*}
  \begin{cases}
    \diff{X(t)} + f(X(t)) \diff{t} \ni b(X(t)) \diff{t} + g(X(t)) \diff{W(t)}, 
    \quad t \in (0,T],\\
    X(0) = X_0,
  \end{cases}
\end{align*}
where $b \colon \R \to \R$ and $g \colon \R \to \R^{1,m}$ are Lipschitz 
continuous and $X_0 \in L^2(\Omega)$. To be more precise, the
piecewise linear 
interpolant $\X$ of the values $(X^n)_{n\in \{0,\dots,N\}}$ defined in 
\eqref{eq6:Xlin} fulfills
\begin{align*}
  \max_{t \in [0,T]} \| X(t) - \X(t) \|_{L^2(\Omega)}
  \leq C k^{\frac{1}{4}}
\end{align*}
for $C \in (0,\infty)$ that does not depend on the step size $k = 
\frac{T}{N}$. However, let us mention that the strong order of convergence of
$1/4$ is not necessarily optimal in this particular example. We refer the
reader to \cite{dareiotis2018} for a corresponding result on the forward
Euler--Maruyama method.

\subsection{Stochastic $p$-Laplace equation}

As a second example, we consider the discretization of the
stochastic $p$-Laplace equation. A 
similar setting is studied in \cite{BreitHofmanova2019}. For a more detailed 
introduction to this class of problems, we refer the reader to this work and 
the references therein.

For $p \in [2,\infty)$ and $T \in (0,\infty)$
the stochastic $p$-Laplace equation is given by
 \begin{align}
  \label{eq7:pLap}
  \begin{cases}
    \diff{}u(t,\xi) - \nabla \cdot \big(|\nabla u(t,\xi)|^{p-2} \nabla u(t,\xi)\big)
    \diff{t} = \Psi(u(t,\xi)) \diff{W(t)},\hspace{-3cm}\\
    &\text{for all } (t,\xi) \in (0,T)\times \D,\\
    u(t,\xi) = 0, &\text{for all } (t,\xi) \in (0,T) \times \partial \D,\\
    u(0,\xi) = u_0(\xi), &\text{for all }  \xi \in \D,
  \end{cases}
\end{align}
where $\D \subset \R^n$, $n \in \N$, is a bounded Lipschitz domain.
By $W \colon [0,T] \times \Omega \to \R^m$, $m \in \N$, we denote a 
standard
$(\F_t)_{t \geq 0}$-adapted Wiener process. We also assume that the initial 
value $u_0 \colon \D \times \Omega \to \R$ fulfills
\begin{align}
  \label{eq7:ini}
  \E \big[ \|u_0\|_{L^2(\D)}^2\big] = \E \Big[ \int_\D |u_0|^2 \diff{\xi}
  \Big] < \infty.
\end{align}
Furthermore, let $\Psi \colon \R \to \L_2(\R^m;\R)$ be a Lipschitz 
continuous mapping, where $\L_2(\R^m;\R)$ denotes the space of 
Hilbert--Schmidt operators from $\R^m$ to $\R$.
Note that the Nemytskii operator $\tilde{\Psi} \colon L^2(\D) \to 
\L_2(\R^m;L^2(\D))$, given by $[\tilde{\Psi}(u)](x) = \Psi(u(x))$ for
$u \in L^2(\D)$, is also 
Lipschitz continuous and will be of importance in the weak formulation 
below.

Further, let $W^{1,p}_0(\D)$ be the Sobolev space of weakly differentiable 
and $p$-fold integrable functions on $\D$ with vanishing trace on the boundary
$\partial \D$, see \cite[Section~1.2.3]{roubicek.2013} or
\cite[Section~4.5]{Winkert2018} for a precise definition. The dual space of
$W_0^{1,p}(\D)$ is denoted by $W^{-1,p}(\D)$ in the following.
Then, the stochastic $p$-Laplace equation \eqref{eq7:pLap} has a 
solution $(u(t))_{t \in [0,T]}$ which is progressively measurable and an
element of $L^2(\Omega; C([0,T]; L^2(\D))) \cap L^p(\Omega;
L^p(0,T;W_0^{1,p}(\D)))$. For further details we refer to
\cite[Example~4.1.9, Theorem~4.2.4]{LiuRoeckner2015}.

For a spatial discretization of \eqref{eq7:pLap}, we use a family of 
finite element spaces $(V_h)_{h>0}$ such that $V_h \subset W_0^{1,p}(\D)$ 
for every $h >0$. Hereby, we interpret  $h$ as a spatial refinement 
parameter. In the following, we consider a fixed parameter value $h >0$.
By $d \in \N$ we then denote the dimension of the space $V_h$.

The spatially semi-discrete problem is to find
a progressively measurable stochastic process $(u_h(t))_{t \in [0,T]}$ 
in the space $L^2(\Omega; C([0,T]; L^2(\D))) \cap 
L^p(\Omega; L^p(0,T;V_h))$ such that 
\begin{align}\label{eq7:FEM}
  \begin{split}
    \int_{\D} u_h(t) v_h \diff{\xi} 
    + \int_{\D} \int_{0}^{t} |\nabla u_h(s)|^{p-2} \nabla u_h(s) \cdot \nabla 
    v_h 
    \diff{s} \diff{\xi}\\
    = \int_{\D} P_h u_0 v_h \diff{\xi} + \int_{\D} \int_{0}^{t} P_h 
    \tilde{\Psi}(u_h(s)) \diff{W(s)} v_h \diff{\xi}
  \end{split}
\end{align}
for every $v_h \in V_h$ and $t\in [0,T]$. Hereby, $P_h \colon L^2(\D) \to V_h$
is the $L^2(\D)$-orthogonal projection onto $V_h$.

In order to apply our results from the previous sections, we rewrite
\eqref{eq7:FEM} as a problem in $\R^d$. To this end, we consider a one-to-one
relation between $V_h$ and $\R^d$ given by
\begin{align}\label{eq7:idRdH}
  v_x = \sum_{i=1}^{d} x_i \varphi_i \in V_h \quad \text{for } x =
  [x_1,\ldots,x_d]^{\top} \in \R^d
\end{align}
for a basis $\{\varphi_1,\dots,\varphi_d\}$ of $V_h$. 
Through \eqref{eq7:idRdH} we induce additional norms on $\R^d$ which are
given by
\begin{align*}
  \|x\|_1 := \|v_x\|_{W_0^{1,p}(\D)},
  \quad
  \|x\|_0 := \|v_x\|_{L^2(\D)},
  \quad
  \|x\|_{-1} := \|v_x\|_{W^{-1,p}(\D)},
\end{align*}
for every $x \in \R^d$. Observe that the norm $\|\cdot\|_0$ 
is also induced by the inner product 
\begin{align*}
  \inner[0]{x}{y} := \inner[L^2(\D)]{v_x}{v_y} = \inner{M_h x}{y}, \quad 
  \text{with } M_h = (\inner[L^2(\D)]{\varphi_i}{\varphi_j})_{i,j \in \{1,\dots,d\}},
\end{align*}
where the mass matrix 
$M_h$ is symmetric and positive definite.
Since all norms on $\R^d$ are equivalent, for each $i \in \{-1,0,1\}$ there exists 
$c_i, C_i \in (0,\infty)$ such that
\begin{align*}
  c_i \|x\|_i \leq |x| \leq C_i \|x\|_i
\end{align*}
for all $x \in \R^d$. 

The $p$-Laplace operator in the spatially semi-discrete problem
\eqref{eq7:FEM} can be written as $A_h \colon V_h \to V_h$ which is implicitly
defined by
\begin{align*}
  \inner[L^2(\D)]{A_h(v_h)}{w_h}
  = \int_{\D} |\nabla v_h|^{p-2} \nabla v_h \cdot \nabla w_h \diff{\xi}
\end{align*}
for all $v_h, w_h \in V_h$. By the same arguments as in
\cite[Example~4.1.9]{LiuRoeckner2015} one can easily verify that $A_h$
fulfills 
\begin{align*}
  &\inner[L^2(\D)]{A_h (v_h) - A_h (w_h)}{v_h - w_h} \geq 0,\\
  &\inner[L^2(\D)]{A_h (v_h)}{v_h} = \|v_h\|_{W_0^{1,p}(\D)}^p, \quad
  \|A_h (v_h)\|_{W^{-1,p}(\D)} \leq \|v_h\|_{W_0^{1,p}(\D)}^{p-1}
\end{align*}
for all $v_h, w_h \in V_h$.
Then, for $x, y \in \R^d$ and associated $v_x, v_y \in V_h$, 
we introduce mappings $\tilde{f} \colon \R^d \to \R^d$ and
$\tilde{g} \colon \R^d \to \R^{d,m}$ implicitly by
\begin{align*}
  \sum_{i=1}^{d} [\tilde{f}(x)]_i \varphi_i = A_h (v_x), \quad
  \sum_{i=1}^{d} [\tilde{g}(x)z ]_i \varphi_i = P_h \tilde{\Psi}(v_x) z,\quad
  \sum_{i=1}^{d} [X_0]_i \varphi_i = P_h u_0
\end{align*}
for $z \in \R^m$ and use these functions to define $f(x) := M_h \tilde{f}(x)$ 
as well as $g(x) := M_h^{\frac{1}{2}}\tilde{g}(x)$ for every $x \in \R^d$.
As we assumed that $v_x \mapsto \tilde{\Psi} (v_x )$ is Lipschitz 
continuous, 
there exists $L_g \in (0,\infty)$ such that
\begin{align*}
  |g(x) - g(y) |^2 
  &= \sum_{j=1}^{m} | M_h^{\frac{1}{2}}\tilde{g}(x)e_j - 
  M_h^{\frac{1}{2}}\tilde{g}(y)e_j|^2\\
  &= \sum_{j=1}^{m} \| P_h \tilde{\Psi}(v_x)e_j - P_h \tilde{\Psi}(v_y)e_j 
  \|_{L^2(\D)}^2\\
  &= \| P_h \tilde{\Psi}(v_x) - P_h \tilde{\Psi}(v_y) \|_{\L_2(\R^m;L^2(\D))}^2\\
  &\leq L_g^2 \| v_x - v_y \|_{L^2(\D)}^2
  \leq \frac{L_g^2}{c_0^2} | x - y |^2
\end{align*}
for $x,y \in \R^d$ and $v_x,v_y \in V_h$ fulfilling \eqref{eq7:idRdH} and
an orthonormal basis $\{e_j\}_{j \in \{1,\dots,m\}}$ of $\R^m$.
Thus, $g$ fulfills Assumption~\ref{ass:g}. 
Due the integrability condition to \eqref{eq7:ini} for $u_0$, it follows that 
$X_0$ fulfills Assumption~\ref{ass:X0}.

Moreover, we see that $f$ is monotone, coercive, and bounded as we can 
write
\begin{align*}
  \inner{f(x) - f(y)}{x - y}
  &= \inner[0]{\tilde{f}(x) - \tilde{f}(y)}{x - y}\\
  &=\sum_{i=1}^{d} \sum_{j=1}^{d} \big([\tilde{f}(x)]_i - [\tilde{f}(y)]_i\big) 
  \big(x_j - y_j\big)
  \inner[L^2(\D)]{\varphi_i }{\varphi_j}\\
  &= \inner[L^2(\D)]{A_h (v_x) - A_h (v_y)}{v_x - v_y}
  \geq 0
\end{align*}
as well as
\begin{align*}
  \inner{f(x)}{x}
  =\sum_{i=1}^{d} \sum_{j=1}^{d} [\tilde{f}(x)]_i x_j
  \inner[L^2(\D)]{\varphi_i }{\varphi_j}
  = \inner[L^2(\D)]{A_h (v_x)}{v_x}
  = \|x\|_1^p
  \geq C_1^{-p} |x|^p
\end{align*}
and
\begin{align*}
  |f(x) | &\leq \|M_h\|_{\L(\R^m)} |M_h^{-1} f(x) |
  \leq C_{-1} \|M_h\|_{\L(\R^m)} \| \tilde{f}(x) \|_{-1} \\
  &= C_{-1} \|M_h\|_{\L(\R^m)} \|A_h (v_x) \|_{W^{-1,p}(\D)} 
  \leq C_{-1} \|M_h\|_{\L(\R^m)} \|v_x\|_{W^{1,p}(\D)}^{p-1} \\
  &= C_{-1} \|M_h\|_{\L(\R^m)} \| x \|_1^{p-1}
  = \frac{C_{-1}}{c_1^{p-1}} \|M_h\|_{\L(\R^m)} |x|^{p-1}
\end{align*}
for all $x,y \in \R^d$ and $v_x,v_y \in V_h$ fulfilling \eqref{eq7:idRdH}. Here, 
$\|\cdot\|_{\L(\R^m)}$ denotes the matrix norm in $\R^m$ which is induced 
by $|\cdot|$.
Therefore, Assumption~\ref{ass:f} is satisfied. To prove 
that $f$ fulfills Assumption~\ref{ass:NochettoIneq} we note that
the mapping $\Phi \colon V_h \to [0,\infty)$ given by
\begin{align*}
  \Phi (v_h) = \frac{1}{p} \int_{\D} |\nabla v_h|^p \diff{\xi}, \quad v_h \in
  V_h,
\end{align*}
is a potential of $A_h$, compare \cite[Example~4.23]{roubicek.2013}. Since
$\Phi$ is convex it follows that
\begin{align*}
  \Phi (v_h) \geq \Phi (w_h) + \inner[L^2(\D)]{A_h(w_h)}{v_h - w_h},
  \quad \text{ for all } v_h, w_h \in V_h,
\end{align*}
where we use \cite[Kapitel~III,~Lemma~4.10]{Gajewski.1974}.
In the same way as in Lemma~\ref{lem:NochettoIneq} we obtain that 
\begin{align*}
  \inner[L^2(\D)]{A_h(v_x) - A_h(v_y)}{v_y - v_z}
  \leq \inner[L^2(\D)]{A_h(v_x) - A_h(v_z)}{v_x - v_z}
\end{align*}
for all $v_z,v_x,v_y \in V_h$. Applying the definition of $f$, we then get
\begin{align*}
  \inner{f(x) - f(y)}{y - z}
  &= \inner[L^2(\D)]{A_h(v_x) - A_h(v_y)}{v_y - v_z}\\
  &\leq \inner[L^2(\D)]{A_h(v_x) - A_h(v_z)}{v_x - v_z}
  = \inner{f(x) - f(z)}{x - z}
\end{align*}
for $x,y,z \in \R^d$ and $v_x,v_y,v_z \in V_h$ fulfilling \eqref{eq7:idRdH}.
This shows that $f$ also fulfills Assumption~\ref{ass:NochettoIneq}. 

Consequently, the results of the previous sections are applicable. More 
precisely, the backward Euler scheme \eqref{eq1:multiEuler} has a unique 
solution $(X^n)_{n\in \{0,\dots,N\}}$ (cf.~Theorem~\ref{thm:existence_multi}).
Theorem~\ref{thm:conv} then states that the piecewise linear interpolant $\X$
of the values $(X^n)_{n\in \{1,\dots,N\}}$ defined in \eqref{eq6:Xlin} fulfills
\begin{align*}
  \max_{t \in [0,T]} \| X(t) - \X(t) \|_{L^2(\Omega;\R^d)}
  \leq C k^{\frac{1}{4}}
\end{align*}
for $C \in (0,\infty)$ that does not depend on the step size $k$ where $X$ 
is the solution to the single-valued stochastic differential equation
\begin{align*}
  \begin{cases}
    \diff{X(t)} + f(X(t)) \diff{t} = g(X(t)) \diff{W(t)}, \quad t \in
    (0,T],\\
    X(0) = X_0.
  \end{cases}
\end{align*}

Observe that our proof does not yet rule out that the constant $C$ above
depends on the dimension $d$ of the finite element space $V_h$. Hence, this is
not a complete analysis of a full discretization of the stochastic partial
differential equation \eqref{eq7:pLap} and a more detailed analysis is subject
to future work. We refer to \cite{BreitHofmanova2019} for a related result in
this direction. 

Let us emphasize that, unlike the results in \cite{BreitHofmanova2019},
we do not have to impose any temporal regularity assumption on the 
exact solution of \eqref{eq7:pLap} or on the solution of the semi-discrete
problem \eqref{eq7:FEM}. Since such regularity conditions are often not easily
verified for quasi-linear stochastic partial differential equations
we are confident that our approach could lead to 
interesting new insights in the numerical analysis of such infinite dimensional
problems.

\section*{Acknowledgment}

ME would like to thank the Berlin
Mathematical School for the financial support. RK also gratefully 
acknowledges
financial support by the German Research Foundation (DFG) through the 
research
unit FOR 2402 -- Rough paths, stochastic partial differential equations and
related topics -- at TU Berlin.

\end{document}